\documentclass[11pt]{amsart}
\usepackage{amsmath, amssymb}
\usepackage{amsfonts}
\usepackage{mathrsfs}
\usepackage[arrow,matrix,curve,cmtip,ps]{xy}

\usepackage{amsthm}
\usepackage{graphicx}

\usepackage{pb-diagram}
 \usepackage{pb-xy}

\usepackage{color}

\allowdisplaybreaks

\newtheorem*{rep@theorem}{\rep@title}
\newcommand{\newreptheorem}[2]{%
\newenvironment{rep#1}[1]{%
 \def\rep@title{#2 \ref{##1}}%
 \begin{rep@theorem}}%
 {\end{rep@theorem}}}
\makeatother

\newreptheorem{theorem}{Theorem}
\newreptheorem{lemma}{Lemma}
\newreptheorem{proposition}{Proposition}
\newreptheorem{corollary}{Corollary}

\newtheorem{theorem}{Theorem}[section]
\newtheorem{lemma}[theorem]{Lemma}
\newtheorem{proposition}[theorem]{Proposition}
\newtheorem{corollary}[theorem]{Corollary}

\newtheorem*{theorem*}{Theorem}
\theoremstyle{remark}
\newtheorem{remark}[theorem]{Remark}
\newtheorem{definition}[theorem]{Definition}

\def\Lem#1{\noindent {\bf Lemma~\ref{L:#1}.}}
\def\Thm#1{\noindent {\bf Theorem~\ref{T:#1}.}}

\def\lem#1{Lemma~\ref{L:#1}}
\def\prop#1{Proposition~\ref{P:#1}}
\def\thm#1{Theorem~\ref{T:#1}}
\def\cor#1{Corollary~\ref{C:#1}}

\def\defn#1{Definition~\ref{D:#1}}
\def\sec#1{Section~\ref{S:#1}}

\newcommand{\Z}{\mathbb{Z}}
\newcommand{\Q}{\mathbb{Q}}
\newcommand{\N}{\mathbb{N}}

\newcommand{\im}{\operatorname{im}}

\newcommand{\aut}{\operatorname{Aut}}

\newcommand{\lup}[1]{\,\,^{#1}\hspace{-.03cm}}
\newcommand{\bd}{\partial}
\newcommand{\w}{\omega}

\def\figr#1{Figure~\ref{F:#1}}

\def\lem#1{Lemma~\ref{L:#1}}
\def\prop#1{Proposition~\ref{P:#1}}
\def\thm#1{Theorem~\ref{T:#1}}
\def\cor#1{Corollary~\ref{C:#1}}

\def\defn#1{Definition~\ref{D:#1}}

\numberwithin{equation}{section}

\begin{document}
\title[A new filtration of the Magnus kernel]{A new filtration of the Magnus kernel of the Torelli group}

\author{R. Taylor McNeill}

\address{Department of Mathematics \\ University of Evansville \\ 1800 Lincoln Ave. \\Evansville, IN 47722\\ USA}
\email{rm217@evansville.edu}


\date{\today}



\begin{abstract}

For a oriented genus $g$ surface with one boundary component, $S$, the Torelli group is the group of orientation
preserving homeomorphisms of $S$ that induce the identity on
homology.  The Magnus representation of the Torelli group represents the action
on $F/F''$ where $F=\pi_1(S)$ and $F''$ is the second term of the
derived series.   We show that the kernel of the Magnus representation, $Mag(S)$,
is highly non-trivial and has a rich structure as a group.
Specifically, we define an infinite filtration of $Mag(S)$ by subgroups,
called the higher order Magnus subgroups, $M_k(S)$.  We develop methods for generating nontrivial mapping classes in $M_k(S)$ for all $k$ and $g \ge 2$.  We show that for each
$k$ the quotient $M_k(S)/M_{k+1}(S)$ contains a subgroup isomorphic to a lower central series quotient of free groups $E(g-1)_k/E(g-1)_{k+1}$.  Finally We show that for $g \ge 3$ the quotient $M_k(S)/M_{k+1}(S)$ surjects onto an infinite rank torsion free abelian group. To do this, we
define a Johnson--type homomorphism on each higher order Magnus
subgroup quotient and show it has a highly non-trivial image.
\end{abstract}

\maketitle

\section{Introduction}
\pagenumbering{arabic} \setcounter{page}{1}

\subsection{Background}

Let $S$ be a closed orientable surface of genus $g$ with 1 boundary component (we will sometimes denote this surface by $S_g$ when it is necessary to be precise about the genus of the surface).  The \emph{mapping class group} of $S$, denoted $Mod(S)$ is the group of isotopy classes of orientation preserving homeomorphisms of $S$ which fix the boundary pointwise.  Two homeomorphisms represent the the same element (called a \emph{mapping class}) if they are isotopic maps where the isotopy also fixes the boundary pointwise.

We study the mapping class group through an analysis of the fundamental group of the surface.  As we restrict to maps which fix the boundary of $S$ pointwise, by choosing a basepoint $*$ on the boundary of $S$, a homeomorphism $f: S \rightarrow S$ induces an automorphism $f_*:\pi_1(S, *) \rightarrow \pi_1(S, *)$.  This correspondence induces a map
$$
Mod(S) \hookrightarrow \aut(\pi_1(S,*))
$$
which yields an injective homomorphism.  It is important to note that for surfaces with boundary, $\pi_1(S,*)$ is a free group.  For convenience of notation, we will henceforth denote this free group by $F$.  Note that $\aut{F}$ is quite complicated.   Hence to effectively employ this homomorphism we study mapping classes which approximate the identity automorphism.  More specifically, we study subgroups of the form $\ker\left( Mod(S) \rightarrow \aut(F/H)\right)$ where $H$ is a characteristic subgroup of $F$.  

Given a group $G$, the \emph{lower central series of G}, $\{G_n\}$ is given inductively by $G_1=G$, $G_k=[G_{k-1},G]$, where $[G_{k-1},G]$ is the subgroup of $G$ generated by elements of the form $aba^{-1}b^{-1}$, $a \in G, \,\, b \in G_{k-1}$.  

The mapping classes which act trivially modulo terms of the lower central series of $F$ form the well--studied Johnson subgroups of the mapping class group.  

\begin{definition}
The $k^{th}$ Johnson subgroup is the subgroup of the mapping class group given by
$J_k(S)=\ker(Mod(S) \rightarrow \aut(F/F_k))$.  
\end{definition}

Note that for $n>k$, as $F_n \subset F_k$ the map from $Mod(S)$ to $\aut(F/F_n)$ factors through $\aut(F/F_k)$:
$$
\xymatrix
{Mod(S) \ar[rr] \ar[rd]& &\aut(F/F_k)\\
&\aut(F/{F_n}) \ar[ru]&}
$$
Hence $\ker(Mod(S) \rightarrow \aut(F/F_n)) \subset \ker(Mod(S) \rightarrow \aut(F/F_k))$ and thus $J_n(S) \subset J_k(S)$.  We achieve a filtration of the Torelli group:
$$
Mod(S)=J_1(S) \supset J_2(S) \supset \cdots \supset J_k(S) \cdots 
$$

The second term of this filtration, $J_2(S)$ is the subgroup of the mapping class group which acts trivially on the homology of $S$.  This subgroup is more commonly known as the Torelli group and frequently denoted $\mathcal{I}$.  The Torelli group plays a crucial role in the study of mapping class groups of surfaces as the quotient $Mod(\Sigma)/ \mathcal{I}(\Sigma)$ is the well understood symplectic group, $Sp(2g, \Z)$.  

It is an easy task to define filtrations of the Torelli group, however the filtration by Johnson subgroups has been integral in their study.  The Johnson subgroup filtration earns its important place in the study of mapping class groups for the many available tools that can be employed for their study.  One class of tools frequently used in exploring the Johnson subgroups is the Johnson homomorphisms.  While these homomorphisms can be defined in a variety of ways, for the course of this paper we find the following definition of the Johnson subgroups to be the most convenient.  

For this definition, it is helpful to note for $f \in J_k(S)$ $f(x)\equiv x \mod F_k$ for all $x \in F$.  Equivalently, $f(x)x^{-1} \in F_k$ for all $x \in F$.  

\begin{definition}
Let $[x] \in H_1(S)$ and let $x$ be an element of the fundamental group realizing the homology class $[x]$.  The $k^{th}$ Johnson homomorphism
$$
\tau_k: J_k(S) \rightarrow Hom(H_1(S), F_k/F_{k+1})
$$
 is given by $\tau_k(f)=([x]\mapsto [f(x)x^{-1}])$. 
\end{definition}

While this definition provides for easy calculations, it does not provide much clarity for why such a homomorphism is well defined.  For a more thorough treatment, see \cite{J1}.

\begin{remark}It is important to note that $\ker \tau_k= J_{k+1}(S)$.  That $\ker \tau_k\supset J_{k+1}(S)$ can be readily seen as for $f \in J_{k+1}(S)$, $f(x)x^{-1} \in F_{k+1}$, and hence $f(x)x^{-1}$ is trivial in $F_k/F_{k+1}$ for all $[x]$.  To see that $\ker \tau_k\subset J_{k+1}(S)$, note that if $f \in \ker \tau_k$, then $f(x)x^{-1}\in F_{k+1}$ for all classes $[x]$.  Thus $f(x)=x \mod F_{k+1}$ and therefore $f \in \ker(Mod(S) \rightarrow \aut(F/F_k))$.  Thus $f \in J_{k+1}(S)$.

This fact provides an enlightening result when performing Johnson homomorphism computations.  If $f$ is an element of $J_k(S)$ such that $\tau_k(f)\ne 0$, then $f \notin J_{k+1}(S)$.  Thus computing $\tau_k(f) \ne 0$ pins the precise location of $f$ in the Johnson filtration to $J_k(S)/J_{k+1}(S)$.
\end{remark}

Of particular interest to the study of mapping class groups is the Magnus representation of the Torelli group, which can be defined as follows.  Given a basis, $\{x_1 , \dots , x_n\}$, for $F$, the Magnus representation of the Torelli group is map which sends a mapping class $f\in Mod(S)$ to a $2g \times 2g$ matrix with entries in $\Z H_1(S)$ namely,
$$
f \mapsto \left(\phi \left(\frac{\partial f(x_i)}{\partial x_j}\right)\right)_{i,j}.
$$
where $\frac{\bd f_*(x_i)}{\bd x_j}$ is the Fox calculus derivative of $f_*(x_i)$ with respect to $x_j$ and  $\phi: \Z[F] \rightarrow \Z[H_1(S)]$ is the natural projection. 
However, the kernel of the Magnus representation, $Mag(S)$ also has a characterization in terms of induced automorphisms \cite{CF}.  Specifically,
$$ 
Mag(S)=\ker \left(Mod(S) \rightarrow \aut(F/F''\right)
$$
where $F''=\left[[F,F],[F,F] \right]$ is the second commutator subgroup of $F$.

While the Magnus representation was first introduced in the 1980s, for many years it was unknown whether the the Magnus representation was a faithful representation of the Torelli group.    This remained an open question until 2001 when Suzuki constructed an explicit mapping class contained in $Mag(S)$ for genus $g \ge 2$  \cite{Suz}.  In 2009 Church and Farb proved that in fact the Magnus kernel is quite large, exhibiting infinitely many independent elements of the Magnus kernel \cite{CF}.  In this paper we demonstrate that $Mag_g$ is larger still, possessing a nontrivial filtration by subgroups, called the higher-order Magnus subgroups,
$$ 
Mag(S)= M_2(S) \supset M_3(S) \supset M_4(S)\supset \cdots 
$$ 
for which the successive quotients are themselves infinitely generated.  The previous examples of Church and Farb are all contained in ${M_2(S)}\setminus{M_3(S)}$.  Hence the higher-order Magnus subgroups reveal new structure in the Magnus kernel.

\subsection{Summary of results}
The Johnson subgroups have provided a key tool for studying the Torelli group.  While there is a clear similarity between the algebraic characterizations of the Magnus kernel and the Torelli group, attempts to define analogous tools for studying the Magnus kernel have been limited.  

For any characteristic subgroup $H$ of $F$, we define an infinite family of subgroups, $J^H_k(S)$, called the higher-order Johnson subgroups.  These subgroups form a filtration of the subgroup $\ker(Mod(S) \rightarrow \aut(F/H))$ of the mapping class group.  The higher-order Johnson subgroup filtration is a generalization of the Johnson subgroup filtration of the Torelli group.  In the special case where $H=[F,F]$, we call these subgroups the higher-order Magnus subgroups, as they yield a filtration of the Magnus kernel.  We show that these higher-order Johnson subgroups have much of the natural structure known for the Johnson subgroups.  These properties include the result that the higher-order Johnson subgroups are equipped with a homomorphism, analogous to the Johnson homomorphisms.

\Thm{welldefined}
\emph{For each characteristic subgroup $H \subset F$ the higher-order Johnson homomorphisms, 
$$\tau^H_k: J^H_k(S) \rightarrow Hom_{\Z [F/H]}(H/H', H_k/H_{k+1}),$$
are well defined, group homomorphisms for $k \ge 2$.}

In the special case of the Magnus subgroups, $M_k(S)$ we give an explicit way of constructing examples in $M_k(S)$ from known examples of mapping classes in $J_k(D)$ where $D$ is a disk with $n$ holes.

\Lem{inclusion}
\emph{  Let $i:D \rightarrow S$ be an embedding such that each boundary component of $i(D)$ is either separating in $S$, or the boundary component of $S$.  Let $[f]\in Mod(D)$ and let $f$ be a homeomorphism representing $[f]$.  Let $f':S \rightarrow S$ be the homeomorphism defined by
$$
f'(x)= \left\{ \begin{array}{ll} f(x) & x \in D\\ x & x \in S \setminus D \end{array} \right.
$$
then if $[f] \in J_k(D)$, $[f'] \in M_k(S)$. }

Using this construction, we describe an explicit subgroup of $M_k(S)/M_{k+1}(S)$ which is isomorphic to a lower central series quotient of free groups.  For $E(n)$ the free group on $n$ generators, we show the following result.

\Thm{containsbraids}
\emph{ Let $S_g$ be an orientable surface with genus $g \ge 3$.  Then the  map $\rho: E(g-1) \rightarrow Mod(S_g)$ induces a monomorphism on the quotients  $\overline{\rho}:E(g-1)_k/ E(g-1)_{k+1} \hookrightarrow M_k(S_g)/M_{k+1}(S_g)$ for all $k$.}

A detailed construction of the map $\rho$ is given in \sec{mainthm}.

Finally, we construct an epimorphism onto an infinite rank torsion free abelian subgroup of $\frac{F'_k}{F'_{k+1}}$, where $F'$ is the commutator subgroup of $F$.  Using Magnus homomorphism computations we prove:

\Thm{infinitelygenerated}
\emph{ Let $S$ be an orientable surface with genus $g \ge 3$.  Then the successive quotients of the Magnus filtration $\frac{M_k(S)}{M_{k+1}(S)}$ surject onto an infinite rank torsion free abelian subgroup of $\frac{F'_k}{F'_{k+1}}$ via the map  
$$
\frac{M_k(S)}{M_{k+1}(S)} \stackrel{\tau'_k(-)[c_6,c_2]}{\longrightarrow}\frac{F'_k}{F'_{k+1}}
$$
where $c_6$ and $c_2$ are generators in the carefully chosen basis for $F$ shown in \figr{infgenbasis}.}

These results establish key tools for working with the Magnus subgroups and unveil new structure in this subgroup of the Torelli group.

\subsection{Outline}

We begin in Section 2 by providing a detailed discussion of generalized Johnson homomorphisms on surfaces with multiple boundary components.  Johnson subgroups of surfaces with multiple boundary components have been employed before, but a precise and detailed treatment of these cases have not yet appeared in the literature.  

In Section 3 we define generalizations of the Johnson subgroups and homomorphisms called the \emph{higher-order Johnson subgroups and homomorphisms}.  A specific case of these generalized Johnson subgroups are the higher-order Magnus subgroups.  These Magnus subgroups provide a filtration of the Magnus kernel and are the central focus of our study.
 
Section 4 contains some group theoretic results that are useful in proving our main theorem.  These results primarily focus on the lower central series quotients of an infinitely generated free group, $E$, and its commutator subgroup, $E'$.  We provide a generalization of the basis theorem for lower central series quotients of free groups which applies to groups which are infinitely generated. We also explore the $\Z[F/F']$ module structure of $F'_k/F'_{k+1}$ use for computing Magnus homomorphisms.

In Section 5 we prove our main results.   We develop a correspondence between Magnus subgroups on surfaces with one boundary component and Johnson subgroups on disks.  We use this correspondence to explore the structure and size of the successive quotients of the higher-order Magnus subgroups $M_k(S)/M_{k+1}(S)$.  We demonstrate that there is a specific subgroup of $M_k(S_g)/M_{k+1}(S_g)$ that is isomorphic to the finitely generated free abelian group $E(g-1)_k/E(g-1)_{k+1}$ where $E(g-1)$ is the free group on $g-1$ generators and $S_g$ is an oriented surface of genus $g$.  We also show that successive quotients of the higher-order Magnus subgroups $M_k/M_{k+1}$ are infinitely generated by displaying a surjection to a infinite rank torsion free abelian group. 

\section{Johnson subgroups and homomorphisms for surfaces with multiple boundary components}
Through the course of this paper we will employ Johnson homomorphisms on surfaces with multiple boundary components.  There are many variations for Johnson subgroups with multiple boundary components.  In addition, there are many cases in which surfaces with multiple boundary components are overlooked in the study of mapping class groups.  Resources detailing definitions and results concerning surfaces with multiple boundary components are sparse difficult to find.  Treatment of Johnson subgroups and Johnson homomorphisms for surfaces with multiple boundary components can be found in \cite{CHH}, \cite{P}, \cite{P1}.  We will take this opportunity to address an analog of the Johnson machinery in detail for surfaces with multiple boundary components, through a perspective compatible with our following definitions of higher-order Johnson subgroups.

Let $\Sigma$ be an orientable surface with $m+1$ boundary components.  Choose an ordering of the boundary components $b_0, \dots, b_m$.  Let $p_i$ be a point on the $i^{th}$ boundary component (we choose $p_0$ to be the basepoint for $\pi_1(\Sigma)$).  Choose arcs $A_i$ which originate from $p_0$ and terminate at $p_i$ for each $0<i<m$.
\begin{definition}\label{D:multipleboundary}
Let $f \in Mod(\Sigma)$.  Then $f$ is in the $k^{th}$ Johnson subgroup of $\Sigma$, $J_k(\Sigma)$ if $f$ satisfies the following two properties:
\begin{enumerate}
\item For $\gamma \in \pi_1(\Sigma)$, $f_*(\gamma)\gamma^{-1}\in \pi_1(\Sigma)_k$. 
\item For all $A_i$, $\left[f(A_i)\overline{A_i}\right] \in \pi_1(\Sigma)_k$
\end{enumerate}
where $\overline{A_i}$ is the reverse of the path $A_i$. 
\end{definition}

Note that when $m=0$ we obtain from this definition the standard Johnson subgroups for a surface with a single boundary component.  Note also that the combination of properties (1) and (2) show that the Johnson subgroups on surfaces with multiple boundary components are independent of the ordering of the boundary components, the choices of points $p_i$ and the choices of arcs $A_i$.

Given this definition of Johnson subgroups on surfaces with multiple boundary components, we would like to be able to easily generate examples of elements in the Johnson subgroups for these surfaces.  Below is a generalization of a result of Morita \cite{M}, which allows us to generate examples in the Johnson subgroups via commutators.

\begin{lemma}\label{L:commutator}
Let $\Sigma$ be an oriented surface with at least one boundary component.  Let $f_k \in J_k(\Sigma)$ and $f_l \in J_l(\Sigma)$.  Then the commutator $[f_k, f_l]$ is contained in $J_{k+l-1}(\Sigma)$.
\end{lemma}

\begin{proof}
It suffices to prove the statement for $k \le l$.  To show that $[f_k, f_l]$ is contained in $J_{k+l-1}(\Sigma)$, we must show the following two conditions are satisfied:

\begin{itemize}
\item[(i)] For each arc $A_i$ connecting the basepoint to the $i^{th}$ boundary component, $[f_k, f_l](A_i) \overline{A_i}\in F_{k+l-1}$.

\item[(ii)] For all $x \in \pi_1(\Sigma), \,\, [f_k,f_l](x) x^{-1} \in F_{k+l-1}$.
\end{itemize}

A result of Morita (\cite{M} Corollary 3.3) shows condition (ii) to be satisfied in the case where $\Sigma$ is a closed surface with a marked point.  In addition, Morita \cite{M} shows when $\Sigma$ is a closed surface with a marked point, $y \in F_l$, $f_{k*}(y)y^{-1} \in F_{k+l-1}$.  While these results are not stated for surfaces with multiple boundary components, the proofs employ only the property that for $\gamma \in \pi_1(\Sigma)$ and $f \in J_n(\Sigma)$, $f_*(\gamma)\gamma^{-1}\in \pi_1(\Sigma)_n$.  As this property also holds for surfaces $\Sigma$ with multiple boundary components, identical arguments show analogous results for the case of multiple boundary components.  We will employ these results for surfaces $\Sigma$ with multiple boundary components with no further proof. 

It suffices to show that  $[f_k, f_l](A_i) \overline{A_i}\in F_{k+l-1}$.  For this we follow the structure of the aforementioned corollary.  As $f_k \in J_k(\Sigma)$, $f_k (A_i) \overline{A_i} \in F_k$.  Let $x_k=f_k(A_i)\overline{A_i}\in F_k$ and note that $f_k(A_i)$ is homotopic rel endpoints to the path $x_k A_i$.  Applying $f_k^{-1}$ to this expression we find $A_i\simeq f_k^{-1}(x_k) f_k^{-1}(A_i) $ or $f_k^{-1}(\overline{x_k})A_i \simeq f_k^{-1}(A_i) $.  Similarly we know $f_l \in J_l(\Sigma)$, $f_l (A_i) \overline{A_i} \in F_l$.  Defining $x_l=f_l(A_i)\overline{A_i}\in F_l$ we have $f_l(A_i)\simeq  x_l A_i$ and $f_l^{-1}(\overline{x_l})A_i \simeq f_l^{-1}(A_i) $.  Using this we can perform the following computation:

\begin{eqnarray}
[f_k,f_l](A_i) &=& f_k f_lf_k^{-1}(f_l^{-1}(A_i)) \nonumber \\
   &\simeq&  f_k f_lf_k^{-1}(f_l^{-1}(x_l^{-1})A_i ) \nonumber \\
   &\simeq& f_k f_l (f_k^{-1} f_l^{-1}(x_l^{-1}) f_k^{-1}(A_i)) \nonumber\\
   &\simeq& f_k f_l (f_k^{-1} f_l^{-1}(x_l^{-1}) f_k^{-1}(x_k^{-1}) A_i) \nonumber\\
   &\simeq& f_k(f_l f_k^{-1} f_l^{-1}(x_l^{-1}) f_l f_k^{-1}(x_k^{-1}) f_l(A_i)) \nonumber\\
   &\simeq& f_k(f_l f_k^{-1} f_l^{-1}(x_l^{-1}) f_l f_k^{-1}(x_k^{-1}) x_lA_i) \nonumber\\
   &\simeq& [f_k,f_l] (x_l^{-1}) f_k f_l f_k^{-1}(x_k^{-1}) f_k(x_l) f_k(A_i) \nonumber\\
   &\simeq& [f_k,f_l] (x_l^{-1}) f_k f_l f_k^{-1}(x_k^{-1}) f_k(x_l) x_k A_i \nonumber
\end{eqnarray}

\bigskip

This gives us the following expression for the homotopy class of the loop $[f_k, f_l](A_i) \overline{A_i}$ in $\pi_1(\Sigma)$:

\begin{eqnarray}
&&\left[[f_k,f_l](A_i)\overline{A_i}\right]   = [f_{k*},f_{l*}] (x_l^{-1}) f_{k*} f_{l*} f_{k*}^{-1}(x_k) f_{k*}(x_l) x_k \nonumber\\
&&= [f_{k*},f_{l*}] (x_l^{-1}) x_l x_l^{-1}f_{k*} f_{l*} f_{k*}^{-1}(x_k^{-1})x_k x_l x_l^{-1} x_k^{-1} f_{k*}(x_l) x_l^{-1}x_k x_l [x_l^{-1},x_k^{-1}]\nonumber
\end{eqnarray}

\bigskip

As $k \le l$, $J_l(\Sigma) \subset J_k(\Sigma)$ and so $[f_k, f_l] \in J_k(\Sigma)$.  As shown in \cite{M}, lemma 3.2 (i), for $y \in F_l$, $f_{k*}(y)y^{-1} \in F_{k+l-1}$.  Thus $[f_{k*},f_{l*}] (x_l^{-1}) x_l \in F_{k+l-1}$.  Looking at this expression mod $F_{k+l-1}$ we have that $[f_k,f_l] (x_l^{-1}) x_l =1$.  By \cite{M} Lemma 3.2 (ii) the class of $f_k f_l f_k^{-1}(x_k^{-1})x_k$ is equal to that of $f_l(x_k^{-1})x_k$ and is thus also in $F_{k+l-1}$ by \cite{M} 3.2 (i).  Similarly, $f_k(x_l) x_l^{-1} \in F_{k+l-1}$. As  $[x_l^{-1},x_k^{-1}] \in [F_l,F_k] \subset F_{k+l}\subset F_{k+l-1}$ this term also reduces to 1 modulo $F_{k+l-1}$.  Since the entire expression is trivial mod $F_{k+l-1}$ it follows that $\left[[f_k, f_l](A_i)\overline{A_i}\right] \in F_{k+l-1}$.
\end{proof}

It is natural to seek an analog for the Johnson homomorphisms which apply to surfaces with multiple boundary components.  Let $\Delta$ be an open arc on $b_0$ originating at $p_0$.  Let $\overline{\Sigma}=\partial\left( \Sigma \times I \right)\setminus \left(\text{int}(\Delta \times I) \right)$.  Note that $\overline{\Sigma}$ is a doubled version of the surface $\Sigma$ with an added boundary component, as illustrated in \figr{sigmabar}.
\begin{figure}[h] 
   \centering
   \includegraphics[width=4.25in]{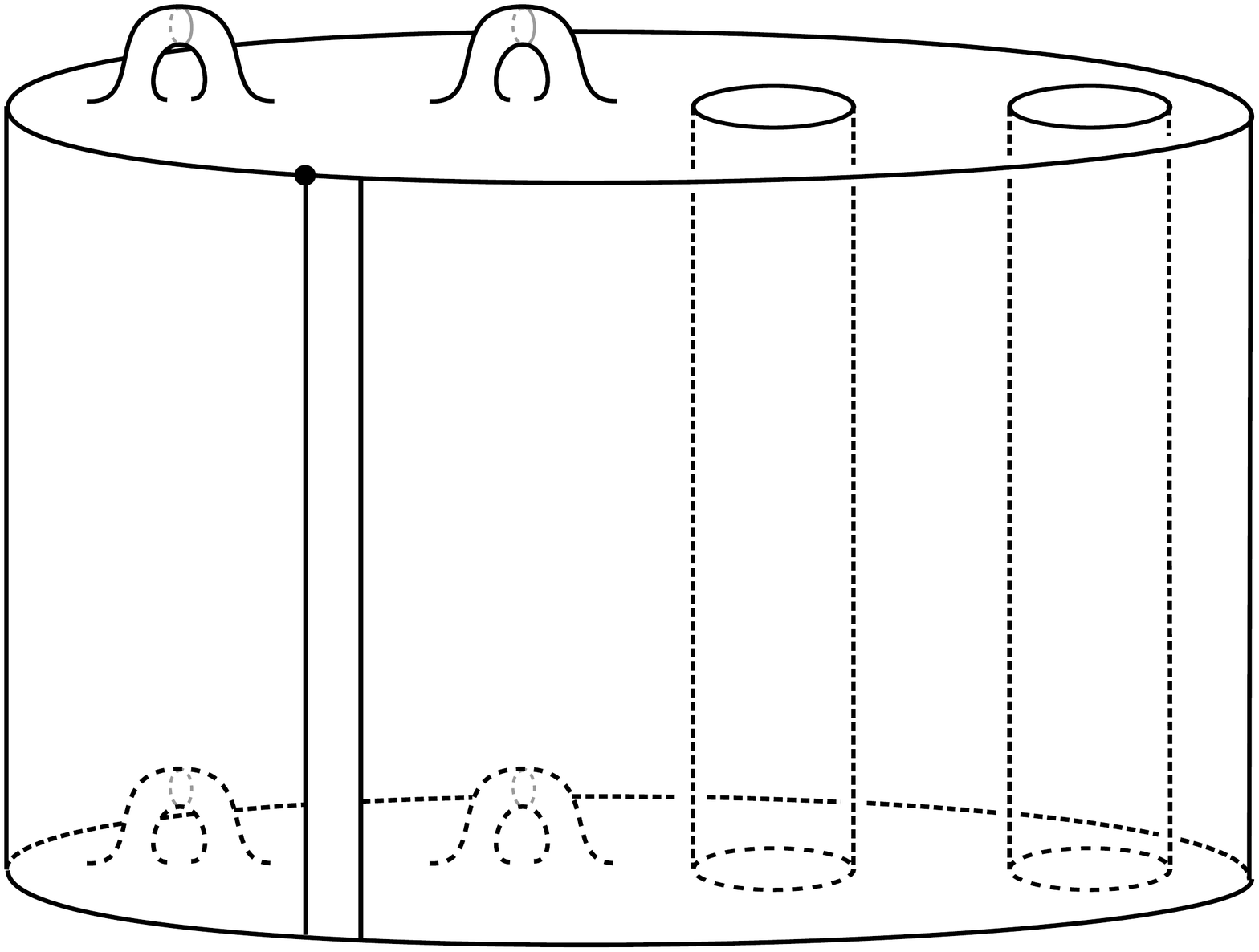}
   \put(-343,183){$\Sigma \times 0= i(\Sigma)$}
   \put(-240,162){$i(p_0)$} 
   \put(-215,174){$\Delta$}
   \put(-213,185){$\dots$}
   \put(-100, 185){$\dots$}
   \caption{An illustration of the doubled surface $\overline{\Sigma}$.}
   \label{F:sigmabar}
\end{figure}
Let $i:\Sigma \rightarrow \overline{\Sigma}$ be the natural embedding which sends $\Sigma$ to $\Sigma \times \{0\}$.  We give $\overline{\Sigma}$ an orientation that agrees with the orientation on $\Sigma$.  We will define the Johnson homomorphisms on $\Sigma$ via the Johnson homomorphisms on $\overline{\Sigma}$.

In order to do this, we first develop some algebraic tools to relate the homology and lower central series quotients of the fundamental groups of $\Sigma$ and $\overline{\Sigma}$.  These build on a  result of Stallings (\cite{S}, Theorem 7.3), reproduced below.  We first define an adaptation of the lower central series: the rational lower central series.  We employ this commutator series to gain insight on the lower central series of free groups.

\begin{definition}
Let $G$ be a group.  The rational lower central series of $G$, with terms $G^r_n$, is defined inductively by setting $G^r_1=G$ and where $G^r_{n+1}$ is the subgroup of $G$ generated by set $S=\{[x,u]|x \in G$, $u \in G^r_{n}\}$ and elements $w$ for which some power of $w$ is a product of elements in $S$. 
\end{definition}

More intuitively, $G^r_{n+1}$ is the smallest subgroup of $G^r_n$ such that $G^r_{n+1}$ is central in $G$ and $G/G^r_{n+1}$ is torsion free.  Note that for a free group $E$, the standard lower central series quotients $E/E_n$ are torsion free.  Thus for a free group $E$, the lower central series of $E$ coincides with its rational lower central series.

\begin{theorem}[Stallings]
If $f: A \rightarrow B$ a homomorphism of abelian groups inducing an isomorphism $f_*: H_1(A;\Q) \rightarrow H_1(B;\Q)$, and a surjective mapping $H_1(A,\Q) \twoheadrightarrow H_1(B,\Q)$.  Then for all finite $n$, $f$ induces isomorphisms
$$
(A^r_{n-1}/A^r_n) \otimes \Q \cong (B^r_{n-1}/B^r_n) \otimes \Q
$$
and for all $k$, $H_k(A/A^r_n) \cong H_k(B/B^r_n)$; $f$ induces embeddings $A/A^r_n \subset B/B^r_n$ and an embedding $A/A^r_{\omega} \subset B/B^r_{\omega}$ at the first infinite ordinal $\omega$. 
\end{theorem}

 We prove the following proposition employing Stalling's result.

\begin{proposition}\label{P:stallings}
Let $A$ and $B$ be groups with $H_2(A;\Q)=H_2(B; \Q)=0$.  Let $h: A \rightarrow B$ be a group homomorphism inducing an injection $H_1(A;\mathbb{Q}) \hookrightarrow H_1(B; \mathbb{Q})$, then for all $n$, $h$ induces an injection $A/A^r_n  \hookrightarrow B/B^r_n$. 
\end{proposition}

\begin{proof}
Consider the injection $h_*:H_1(A;\mathbb{Q}) \rightarrow H_1(B, \mathbb{Q})$.  As $H_1(B, \mathbb{Q})$ is a $\mathbb{Q}$ vector space, it decomposes as $H_1(B; \mathbb{Q}) \cong H_1(A; \mathbb{Q}) \oplus V$ where $V$ is a $\Q$ vector space. Let $C$ be a free group of the same rank as $V$ with generating set $\{c_i\}$ and note that $H_1(C; \mathbb{Q}) \cong V$.   Let $\{e_i\}$ be a basis for $V$ and choose elements $b_i \in B$ such that $b_i \mapsto e_i$ through the isomorphism $H_1(B; \mathbb{Q}) \cong H_1(A; \mathbb{Q}) \oplus V$.  There is a unique group homomorphism $g:C \rightarrow B$ such that $c_i \mapsto b_i$.  Consider the map $h* g: A * C \rightarrow B$.  By construction, this is a group homomorphism which induces an isomorphism $(h * g)_* : H_1(A * C; \mathbb{Q})\rightarrow  H_1(B; \mathbb{Q})$.  As $H_2(A*C)=H_2(B)=0$, clearly the induced map $H_2(A*C;\Q) \rightarrow H_2(B;\Q)$ is surjective.  Hence by Stallings result, for all $n$, $A * C/(A * C)^r_n \stackrel{(h*g)_*}{\cong} B/B^r_n$.  As 
$$
A/A^r_n \hookrightarrow A/A^r_n * C/C^r_n \cong A * C/(A * C)^r_n \stackrel{(h*g)_*}{\cong} B/B^r_n,
$$
the map $A/A^r_n \rightarrow B/B^r_n$ induced by $h$ is injective.
\end{proof}

\begin{remark}
Note that for a free group $E$, since $E$ is torsion free, the rational lower central series agrees with the standard lower central series, i.e. $E^r_n=E_n$.  Hence for free groups $A$ and $B$ satisfying the conditions of \prop{stallings} we achieve an injection $A/A_n  \hookrightarrow B/B_n$ on the standard lower central series quotients.  We will make extensive use of this fact throughout the paper.
\end{remark}

For ease of notation, let us rename $C=\pi_1(\Sigma, p_0)$ and $\overline{C}=\pi_1(\overline{\Sigma}, i(p_0))$.

\begin{lemma}\label{L:sigmastallings}
The embedding $i: \Sigma \rightarrow \overline{\Sigma}$ induces a group monomorphism $$\overline{i_*}:\frac{C_k}{C_{k+1}}\rightarrow \frac{\overline{C}_k}{\overline{C}_{k+1}}.$$
\end{lemma}

\begin{proof}
This is a direct application of \prop{stallings}.  Note that as $\Sigma$ and $\overline{\Sigma}$ are surfaces with boundary, they each deformation retract to a wedge of circles.  Thus $\pi_n(\Sigma)=\pi_n(\overline{\Sigma})=1$ for $n>1$.  Thus $\Sigma$ is a $K(C,1)$ and $\overline{\Sigma}$ is a $K(\overline{C},1)$.  Hence $H_2(C,\Q)=H_2(\Sigma,\Q)=0$ and $H_2(\overline{C},\Q)=H_2(\overline{\Sigma},\Q)=0$.  The embedding $i$ induces a homomorphism $C \rightarrow \overline{C}$ and a monomorphism $i_*:H_1(C;\Q) \rightarrow H_1(\overline{C};\Q)$.    Thus $C$ and $\overline{C}$ satisfy the conditions of \prop{stallings} and hence we achieve an injective homomorphism $\overline{i_*}:\frac{C_k}{C_{k+1}}\rightarrow \frac{\overline{C}_k}{\overline{C}_{k+1}}.$
\end{proof}

We now work to relate the homology of $\Sigma$ and $\overline{\Sigma}$.  Let $\Theta=\overline{\Sigma} \setminus \text{int}\left( i(\Sigma) \right)$ and let $j:\Theta \rightarrow \overline{\Sigma}$ be the natural inclusion map.  The inclusion $j$ yields the following long exact sequence of a pair:
$$
\cdots \rightarrow H_1(\Theta) \stackrel{j_*}{\rightarrow}H_1(\overline{\Sigma}) \stackrel{\pi}{\rightarrow} H_1(\overline{\Sigma}, \Theta) \rightarrow \stackrel{\sim}{H_0}(\Theta) \rightarrow \stackrel{\sim}{H_0}(\overline{\Sigma}).
$$
Note in particular that this exact sequence provides us with an isomorphism 
$$
\pi: \frac{H_1(\overline{\Sigma})}{j_*(H_1(\Theta))} \stackrel{\cong}{\rightarrow} H_1(\overline{\Sigma}, \Theta).
$$
By excision, the inclusion $i: \Sigma \rightarrow \overline{\Sigma}$ induces an isomorphism on homology: 
$$
i_*:H_1(\Sigma, \partial \Sigma) \stackrel{\cong}{\rightarrow} H_1(\overline{\Sigma}, \Theta).
$$
Hence there is an isomorphism 
$$
\pi^{-1}i_*:H_1(\Sigma, \partial \Sigma) \stackrel{\cong}{\rightarrow} \frac{H_1(\overline{\Sigma})}{j_*(H_1(\Theta))}.
$$

Let $[f] \in J_k(\Sigma)$ and let $f$ be a representative homeomorphism of $[f]$.  Let  $\overline{f}: \overline{\Sigma} \rightarrow \overline{\Sigma}$ be given by 
$$
\overline{f}(x)= \begin{cases} f(x) \qquad & \text{if }x \in \Sigma\\ x & \text{if } x \notin \Sigma \end{cases}.
$$
Let $i':Mod(\Sigma) \rightarrow Mod(\overline{\Sigma})$ be the map given by $i'([f])=\left[\overline{f}\right]$.  Note that this map is well defined since isotopic maps on $\Sigma$ extend to isotopic maps on $\overline{\Sigma}$.  It is naturally a homomorphism.

Let $\eta_k$ be the map
$$
\eta_k: Hom\left( \frac{H_1(\overline{\Sigma})}{j_*(H_1(\Theta))}, \frac{\overline{C}_k}{\overline{C}_{k+1}}\right) \rightarrow  Hom\left( {H_1({\Sigma}, \partial \Sigma)}, \frac{\overline{C}_k}{\overline{C}_{k+1}}\right)
$$
which is the dual of the isomorphism $\pi^{-1}i_*$.

\begin{lemma}\label{L:excision}
Given a mapping class $[f] \in Mod(\Sigma)$, $\tau_k\left(i'([f])\right) \in  Hom\left( \frac{H_1(\overline{\Sigma})}{j_*(H_1(\Theta))}, \frac{\overline{C}_k}{\overline{C}_{k+1}}\right)$.  Furthermore, for $[\alpha] \in H_1(\Sigma, \partial \Sigma)$, 
$$
\eta_k \tau_k i'\left([f]\right)[\alpha] \in \overline{i}\left( \frac{C_k}{C_{k+1}}\right).
$$

Equivalently, we have the following sequence of maps
\begin{align*}
J_k(\Sigma) \stackrel{i'}{\rightarrow}  J_k(\overline{\Sigma}) \stackrel{\tau_k}{\rightarrow} Hom\left( \frac{H_1(\overline{\Sigma})}{j_*(H_1(\Theta))}, \frac{\overline{C}_k}{\overline{C}_{k+1}}\right)& \stackrel{\eta_k}{\rightarrow}  Hom\left( {H_1({\Sigma}, \partial \Sigma)}, \frac{\overline{C}_k}{\overline{C}_{k+1}}\right) \\
& \qquad \qquad \stackrel{\overline{i}^{-1}}{\rightarrow}Hom\left( {H_1({\Sigma}, \partial \Sigma)}, \frac{C_k}{C_{k+1}}\right).
\end{align*}
\end{lemma}

\begin{proof}
To prove the first statement in the lemma we examine an element $j_*[\beta] \in H_1(\overline{\Sigma})$.  The element $[\beta] \in H_1(\Theta)$ has a representative element $\beta \in \pi_1(\Theta)$.  As the following diagram commutes
$$
\xymatrix
{\pi_1(\Theta) \ar[r]^{j_*} \ar[d]&\pi_1(\Sigma)\ar[d]\\
H_1(\Theta) \ar[r]^{j_*} & H_1(\Sigma)}
$$
we have that $j_*[\beta]\in H_1(\overline{\Sigma})$ has a representative loop $\beta$ which lies entirely in $\Theta$.  Thus by definition of $i'$, for any $[f]\in J_k(\Sigma)$, $\tau_k\left(i'\left([f]\right) \right)[\beta]=\overline{f}(\beta)\beta^{-1}=\beta \beta^{-1}=1$.  Hence $\tau_k\left(i'([f])\right) \in  Hom\left( \frac{H_1(\overline{\Sigma})}{j_*(H_1(\Theta))}, \frac{\overline{C}_k}{\overline{C}_{k+1}}\right)$.

To prove that $\eta_k \tau_k i'\left([f]\right)[\alpha] \in \overline{i}\left( \frac{C_k}{C_{k+1}}\right)$ let us consider the following basis for $H_1(\Sigma)$, shown in \figr{sigmabasis}.

\begin{figure}[h] 
   \centering
   \includegraphics[width=4.25in]{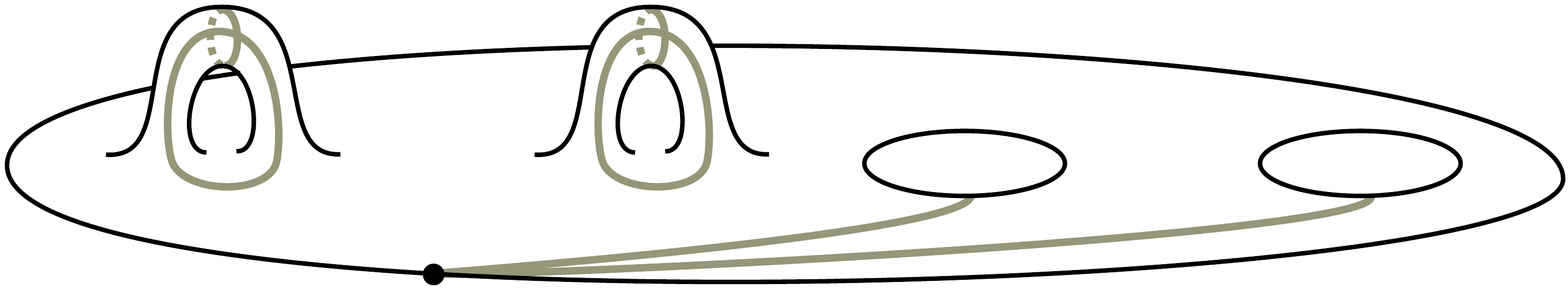}
   \put(-269,173){$a_1$}
   \put(-269,220){$a_2$} 
\put(-189,220){$a_{2g-1}$} 
\put(-189,173){$a_{2g}$} 
   \put(-227,185){$\dots$}
   \put(-87, 185){$\dots$}
\put(-152,175){$A_1$}
\put(-77,175){$A_n$}
   \vspace{-5.5cm}
\caption{A basis for the relative homology $H_1({\Sigma}, \partial \Sigma)$.}
   \label{F:sigmabasis}
\end{figure}

Through the map $\pi^{-1}i_*$ these basis elements map to loops in $H_1(\overline{\Sigma})$ as shown in \figr{doubledsurfacebasis}.

\begin{figure}[h] 
   \centering
   \includegraphics[width=5.0in]{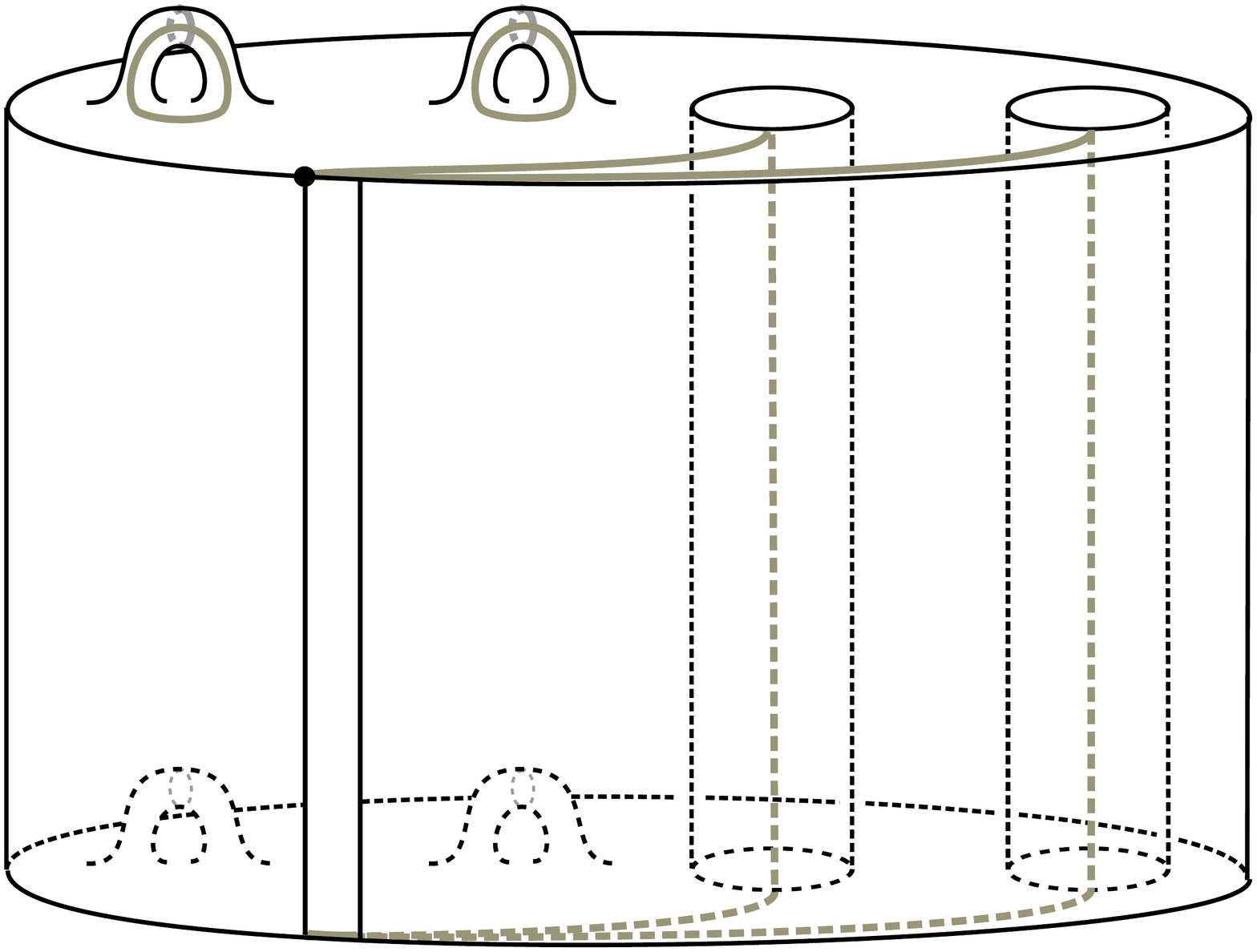}
   \put(-290,208){$a_1$}
   \put(-290,246){$a_2$} 
\put(-218,208){$a_{2g-1}$} 
\put(-213,246){$a_{2g}$} 
\put(-182,210){$A_1$}
\put(-108,211){$A_n$}
   \put(-249,218){$\dots$}
   \put(-120,218){$\dots$}
   \caption{The elements of $H_1(\overline{\Sigma})$ corresponding to the basis of  $H_1({\Sigma}, \partial \Sigma)$ chosen in \protect\figr{sigmabasis}.}
   \label{F:doubledsurfacebasis}
\end{figure}
Note that the loops $a_i$ include to the same homology elements of $H_1(\Sigma)$.  Let the arcs $A_i$ be parametrized by $t \in [0,1]$.  Then under this map the arcs $A_i$ are sent to loops $c_i$ given by:
$$
c_i(t)=\begin{cases} 
(A_i(4t), 0)  \qquad &0 \le t \le 1/4\\
(p_i,4t-1) & 1/4 < t < 1/2\\
(A_i(3-4t), 1) & 1/2 \le t < 3/4\\
(p_0,4-4t)& 3/4 \le t \le 1\\
\end{cases}
$$
as illustrated in \figr{doubledsurfacebasis}.  First, note that for any homology class $[\alpha] \in H_1(\overline{\Sigma})$ which has a representative loop $\alpha \in i_* \pi_1(\Sigma)$ and for any $[f]\in J_k(\Sigma)$ we have
\begin{align*}
\tau_k(i'([f])[\alpha]&=[\overline{f}(\alpha) \alpha^{-1}]\\
&=[i_*f(\alpha) \alpha^{-1}]\\
&=\overline{i_*}[f(\alpha)\alpha^{-1}]
\end{align*}
Hence the image by $\tau_k(i'([f])$ of $a_i$ yields an element of $\overline{i_*}\left( \frac{C_k}{C_{k+1}}\right)$. 

Let $B_i$ be the segment of $c_i$ parametrized by $1/4\le t \le 1$ so that $c_i=A_i\cup B_i$.  Note that by definition  $i'({f})$ acts by the identity on $B_i$ and acts by $f$ on $A_i$.  By construction the loops $c_i$ are based at $p_0$.  Thus they also represent elements of $\pi_1(\overline{\Sigma})$. We will abuse notation by referring to the parametrized loop, the homotopy class, and the homology class of $c_i$ as $c_i$.   Then we may compute $\tau_k(i'([f])c_i$ as follows.
\begin{align*}
\tau_k(i'([f])c_i&=[\overline{f}(c_i) c_i^{-1}]\\
&=[\overline{f}(A_i  B_i) \overline{(A_iB_i)}]\\
&=[i(f(A_i))B_i \overline{(A_iB_i)}]\\
\qquad&=[i(f(A_i))B_i \overline{B_i}\,\, \overline{A_i}]\\
&=[i(f(A_i))\overline{A_i}]\\
\qquad&=[i_*(f(A_i)\overline{A_i})]\\
&=\overline{i_*}[(f(A_i)\overline{A_i})]\\
\end{align*}
This shows that for each $i$, $\tau_k(i'([f])c_i \in \overline{i_*}\left( \frac{C_k}{C_{k+1}}\right)$.  As \linebreak $\tau_k(i'([f])[\alpha] \in \overline{i_*}\left( \frac{C_k}{C_{k+1}}\right)$ for all elements $[\alpha]$ of a basis for $H_1(\Sigma, \bd \Sigma)$, then for all $[\alpha] \in H_1(\Sigma, \bd \Sigma)$ we have $\tau_k(i'([f])[\alpha] \in \overline{i_*}\left( \frac{C_k}{C_{k+1}}\right)$.

This shows we have the following composition of homomorphisms:
\begin{align*}
J_k(\Sigma) \stackrel{i'}{\rightarrow}  
J_k(\overline{\Sigma}) \stackrel{\tau_k}{\rightarrow} 
&Hom\left( \frac{H_1(\overline{\Sigma})}{j_*(H_1(\Theta))}, \frac{\overline{C}_k}{\overline{C}_{k+1}}\right)\\ &\qquad \qquad \qquad \qquad \quad \stackrel{\eta_k}{\rightarrow} 
 Hom\left( {H_1({\Sigma}, \partial \Sigma)}, \overline{i_*}\left(\frac{{C}_k}{{C}_{k+1}}\right)\right).
\end{align*}
By applying $\overline{i_*}^{-1}$ on the range of $Hom\left( {H_1({\Sigma}, \partial \Sigma)}, \overline{i_*}\left(\frac{{C}_k}{{C}_{k+1}}\right)\right)$ we get the following composition:
\begin{align*}
J_k(\Sigma) \stackrel{i'}{\rightarrow}  J_k(\overline{\Sigma}) \stackrel{\tau_k}{\rightarrow} Hom\left( \frac{H_1(\overline{\Sigma})}{j_*(H_1(\Theta))}, \frac{\overline{C}_k}{\overline{C}_{k+1}}\right)& \stackrel{\eta_k}{\rightarrow}  Hom\left( {H_1({\Sigma}, \partial \Sigma)}, \frac{\overline{C}_k}{\overline{C}_{k+1}}\right) \\
&  \stackrel{\overline{i_*}^{-1}}{\rightarrow}Hom\left( {H_1({\Sigma}, \partial \Sigma)}, \frac{C_k}{C_{k+1}}\right)
\end{align*}
as desired.
\end{proof}

\begin{definition}
We define the generalized Johnson homomorphisms for surfaces with multiple boundary components 
$$
\tau_k: J_k(\Sigma) \rightarrow Hom\left( H_1({\Sigma}, \partial \Sigma), \pi_1({\Sigma})_k/ \pi_1({\Sigma})_{k+1} \right)
$$
 to be the composition $\overline{i}^{-1} \eta_k \tau_k i'$ given in \lem{excision}.
\end{definition}

Thus to compute the Johnson homomorphism for surfaces with multiple boundary components, we must consider how the mapping class acts on all representatives of a basis for $Hom\left( H_1({\Sigma}, \partial \Sigma) \right)$.  In particular, this includes the action on arcs joining boundary components of $\Sigma$.  It suffices to consider the action of mapping classes on arcs $A_i$ (as described in \defn{multipleboundary}).   As shown in the proof of \lem{excision}, for these arcs we obtain the Johnson homomorphism $\tau_k([f])=\left([A_i]\mapsto [f(A_i) \overline{A_i}]\right)$.   Note that \defn{multipleboundary} verifies that $f(A_i) \overline{A_i}$ is in fact an element of $\pi_1(\Sigma)_k$ as desired.

\section{Higher-Order Johnson Subgroups and Homomorphisms}

\subsection{Higher-Order Johnson Subgroups and Homomorphisms}\label{S:higherjohnson}
The Johnson subgroups and homomorphisms are heavily built upon the lower central series.  In this section we generalize the concepts of Johnson subgroups and homomorphisms to more general characteristic subgroups.  These tools are useful in analyzing subgroups of the mapping class group which induce trivial automorphisms on $F/H$ for any characteristic subgroup~$H$.

Recall that $S$ is an oriented surface with one boundary component and let $*$ be a basepoint for $\pi_1(S)$ which lies on the boundary.  We are then able to define the higher-order Johnson subgroups as follows.

\begin{definition}
Let $F=\pi_1(S, *)$ and let $H$ be a characteristic subgroup of $F$.  Let $\phi^H: Mod(S)\rightarrow \aut(F/H)$ be the map which takes a homeomorphism class in $Mod(S)$ to the induced automorphism of $F/H$.  We define the higher-order Johnson subgroup with characteristic subgroup $H$, $J^{H}(S)$, by $J^{H}(S)=\ker \phi^H$.  Equivalently, $J^{H}(S)$ is the subgroup of the mapping class group which acts trivially on $F/H$.

For any characteristic subgroup $H$ of $F$, $H_k$ is also a characteristic subgroup of $F$.  The higher-order Johnson subgroup with characteristic subgroup $H_k$ is denoted $J^H_k(S)$.  As any homeomorphism acting trivially on $F/H_k$ also acts trivially on $F/H_n$ for $n<k$, $J^H_k(S) \subset J^H_n(S)$ for $n<k$.  Hence the subgroups $J^H_k(S)$ form a filtration of $J^H(S)$: the higher-order Johnson filtration with characteristic subgroup $H$.
$$
J^H(S)=J^H_1(S) \supset J^H_2(S) \supset J^H_3(S) \supset \cdots \supset J^H_k(S) \supset \cdots,
$$   
\end{definition}

Note that the traditional Johnson filtration is recovered by choosing $H=F$.

There is a natural structure on $H_k/H_{k+1}$ as a left $\Z[F/H]$ module for all $k \ge 1$.  Here the module action by elements  $[g] \in F/H$ is given by $[g]\cdot[x]=[gxg^{-1}]$.  It is clear that this action is well defined since given $g \in H$ and $x \in H_k$, the commutator $gxg^{-1}x^{-1}$ belongs to $H_{k+1}$.  Hence for $g \in H$,  $[gxg^{-1}]=[x]$ as elements of $H_k/H_{k+1}$.  The action by elements of $\Z[F/H]$ is given by the obvious extension.  It is important to note that often $F/H$ is a nonabelian group, and hence $H_k/H_{k+1}$ is a module over a noncommutative ring.

Having constructed subgroups analogous to the Johnson subgroups, it is natural to develop a corresponding analog to the Johnson homomorphisms.

\begin{definition}
The higher-order Johnson homomorphisms, 
$$
\tau^H_k(f): J^H_k(S) \rightarrow Hom_{\mathbb{Z}\left[F/H\right]}(H/H' \rightarrow H_k/H_{k+1}),
$$
are given by $\tau^H_k(f)=\left([x] \mapsto [f_*(x)x^{-1}]\right)$ where $f_*:F \rightarrow F$ is the automorphism induced by $f$.   
\end{definition}

\begin{theorem}\label{T:welldefined}
The higher-order Johnson homomorphisms, 
$$\tau^H_k: J^H_k(S) \rightarrow Hom_{\Z [F/H]}(H/H', H_k/H_{k+1}),$$
are well defined, group homomorphisms for $k \ge 2$.
\end{theorem}

\begin{proof}
We will start by showing that for each $f \in J^H_k(S)$ the map $\tau^H _k(f)$ is a well defined $\mathbb{Z}\left[F/H\right]$-module homomorphism.  We first show that for $[a,b]=aba^{-1}b^{-1}$, where $a,b \in H$, $\tau^H_k(f)\left([a,b]\right)=0$ in $H_k/H_{k+1}$.  By definition, 

\begin{align*}
\tau^H_k(f)\left([a,b]\right)&=f_*\left([a,b]\right)[a,b]^{-1}\\
&=[f_*(a),f_*(b)][a,b]^{-1}\\
&=[ad,be][a,b]^{-1} \qquad \qquad \text{for some }d,e \in H_k.\\
\end{align*}

Using the commutator identities $[ux,y]= \lup{u}[x,y][u,y]$ and $[x,vy]=[x,v] \lup{v}[x,y]$, where $\lup{h}g=hgh^{-1}$, we can simplify this further.

\begin{align*}
\tau^H_k(f)\left([a,b]\right)&=\lup{a}[d,be][a,be][a,b]^{-1}\\
&= \lup{a}[d,b] \lup{ab}[d,e][a,b] \lup{b}[a,e][a,b]^{-1} \qquad \quad
\end{align*}

As $d,e \in H_k$, $[d,b], [d,e],[a,e] \in H_{k+1}$.  Therefore, this expression is trivial in the quotient $H_k/H_{k+1}$.

We will next show that $\tau^H_k(f)$ is multiplicative.  By definition, for $a,b \in H$,

\begin{align*}
\tau^H_k(f)(ab)&=f_*(ab)(ab)^{-1}&\\
&=f_*(a)f_*(b)b^{-1}a^{-1}&\\
&=f_*(a)a^{-1} \lup{a} \left(f_*(b)b^{-1}\right)&\\
&=f_*(a)a^{-1}f_*(b)b^{-1} \qquad &\text{as in $H_k/H_{k+1}$, conjugation by}\\
& & \text{an element in $H$ is trivial.}\\
&=\tau^H_k(f)(a) \tau^H_k(f)(b).&
\end{align*}

Any $w \in [H,H]$ can be written as a product of commutators  \linebreak $w=c_1 \cdots c_n$.  This completes the proof that $\tau^H_k(f)$ is well defined, as  $\tau^H_k(f)(w)=\tau^H_k(f)(c_1) \cdots \tau^H_k(f)(c_n)=0$.  This also shows that $\tau^H_k(f)$ is a group homomorphism.

To show that $\tau^H_k(f)$ is a module homomorphism for a given $f$ we must show for $[g] \in F/H$ and $[x] \in H/H'$, $[g] \cdot \tau^H_k(f)([x])=\tau^H_k(f)( [g] \cdot [x])$.  As the module action is by conjugation, we may compute as follows.

\begin{align*}
\tau^H_k(f)([g] \cdot [x])&= \tau^H_k(f)\left([gxg^{-1}]\right)\\
&=\left[f_*(gxg^{-1})(gxg^{-1})^{-1}\right]\\
&=\left[f_*(g)f_*(x)f_*(g^{-1})gx^{-1}g^{-1}\right]\\
&=\left[f_*(g)g^{-1}gf_*(x)g^{-1}gf_*(g^{-1})gx^{-1}g^{-1}\right]
\end{align*}
This expression reduces to:
\begin{align*}
\tau^H_k(f)([g] \cdot [x])&=\left[ (f_*(g)g^{-1})gf_*(x)g^{-1}   (f_*(g)g^{-1})^{-1}   gx^{-1}g^{-1}\right]\\
\end{align*}

The element $g f_*(x)g^{-1} \in H$ as $H$ is a characteristic subgroup.  As \linebreak $f \in J^H_k$, $f_*$ acts trivially mod $H_k$, and thus $f_*(g)g^{-1}\in H_k$.  Since \linebreak $\tau^H_k(f)\left([x] \cdot [f]\right) \in H_k/H_{k+1}$, the conjugation of an element of $H$ by an element of $H_k$ is a trivial conjugation.  This observation yields the following expression. 

\begin{align*}
\tau^H_k(f)([x]\cdot [g])&=\left[gf_*(x)g^{-1}   gx^{-1}g^{-1}\right]\\
&=\left[gf_*(x)x^{-1}g^{-1}\right]\\
&=[g] \cdot \tau^H_k(f)(x)
\end{align*}

This concludes the proof that $\tau^H_k(f)$ is an $\Z[F/H]$-module homomorphism.  It remains to show that $\tau^H_k:J^H_k(S) \rightarrow Hom_{\Z[F/H]}(H/H', H_k/H_{k+1})$ is a group homomorphism.

Let $f^1,f^2 \in J^H_k(S)$ and let $x \in H/H'$.  We consider the image of their product by the map $\tau^H_k$ in the computation below.

\begin{align*}
\tau^H_k(f^1f^2)(x)&=(f^1f^2)_*(x)x^{-1}\\
&=f^1_*f^2_*(x)x^{-1}\\
&=f^1_*(f^2_*(x))x^{-1}\\
&=f^1_*(f^2_*(x))(f^2_*(x))^{-1} f^2_*(x)x^{-1}\\
&=\tau^H_k(f^1)(f^2_*(x))\, \tau^H_k(f^2)(x)
\end{align*} 

As $f^2 \in J^H_k(S)$, $f^2_*(x)=x$ as an element of $H/H'$ for $k\ge 2$.  Hence  $\tau^H_k(f^1)(f^2_*(x))=\tau^H_k(f^1)(x)$.  Combining this with the above computation  gives us the desired result: $\tau^H_k(f^1f^2)(x)=\tau^H_k(f^1)(x) \, \tau^H_k(f^2)(x)$.  Thus  $\tau^H_k:J^H_k(S) \rightarrow Hom_{\Z[F/H]}(H/H', H_k/H_{k+1})$ is a group homomorphism.
\end{proof}

\begin{proposition}\label{P:kernelcontained}
$J^H_{k+1} \subset \ker \tau^H_k$.  Thus 
$$
\tau^H_k: \frac{J^H_{k}}{J^H_{k+1}} \rightarrow Hom_{\Z [F/H]}(H/H', H_k/H_{k+1})
$$
is a well defined map.
\end{proposition}

\begin{proof}
By definition, $J^H_{k+1}= \ker\left(Mod(S) \rightarrow \aut (F/H_{k+1})\right)$.  Thus for \linebreak $[f] \in J^H_{k+1}, [x] \in F/H$, we have for any representative homeomorphism \linebreak $f \in [f]$, $f_*(x)=x \mod H_{k+1}$.  Rewriting this expression we see \linebreak $f_*(x)x^{-1} \in H_{k+1}$.  Thus $ \left[f_*(x)x^{-1}\right]=1$ as an element of $H_k/H_{k+1}$.  Hence $[f]\in \ker \tau'_k$.
\end{proof}

\subsection{Higher-order Magnus Subgroups}

While the higher-order Johnson subgroups and homomorphisms are defined for any characteristic subgroup $H$, this machinery is of particular interest in the case where $H$ is the commutator subgroup of $F$, denoted by $[F,F]$ or $F'$.  Through the remainder of the paper, we focus primarily on this case.  For clarity, we repeat the definitions of the higher-order Johnson subgroups and homomorphisms here for this special case. 

\begin{definition}
For $k \ge 2$, the higher-order Magnus subgroups $M_k(S)$ are given by $M_k(S)=J^{[F,F]}_k(S)$.  Equivalently, 
$$
M_k(S)= \ker\left( Mod(S) \rightarrow \aut(F/F'_k)\right)
$$
\end{definition}

It is of particular importance that $M_1(S)=\ker(Mod(S) \rightarrow \aut(F/F'_{1})$ is the Torelli group, and $M_2(S)=\ker(Mod(S) \rightarrow \aut(F/F'_{2})$ is the kernel of the Magnus representation of the Torelli group.  Thus the higher-order Magnus filtration,
$$
Mag(S)=M_2(S) \supset M_3(S) \supset \cdots \supset M_k(S) \supset \cdots
$$
 is a filtration of the Magnus kernel.

To investigate the structure of these higher-order Magnus subgroups, we will make frequent use of their corresponding higher-order Johnson homomorphisms.

\begin{definition}
The higher-order Magnus homomorphisms, 
$$
\tau'_k(f): M_k(S) \rightarrow Hom_{\mathbb{Z}\left[F/F'\right]}(F'/F'' \rightarrow F'_k/F'_{k+1}),
$$
are the higher-order Johnson homomorphisms with characteristic subgroup $F'$.   
\end{definition}

\begin{remark}
Note that as a special case of \prop{kernelcontained} we have that $M_{k+1} \subset \ker \tau'_k$.  Hence the Magnus homomorphisms are well defined on successive quotients
$$
\tau'_k: \frac{M_{k}}{M_{k+1}} \rightarrow Hom_{\Z [F/F']}(F'/F'', F'_k/F'_{k+1}).
$$
Thus, just as with the Johnson homomorphisms, for $f \in M_k$, computing $\tau'_k(f) \ne 0$ pins the location of $f$ in the higher-order Magnus filtration precisely.
\end{remark}

\section{Algebraic Tools}

We take this opportunity to prove some algebraic results that will be of use in proving our main theorems.  

\subsection{Basis theorems and properties of lower central series quotients}\label{S:basistheorems}

We will make extensive use of several variations on the basis theorem for lower central series quotients of free groups \cite{H} Theorem 11.2.4.  We begin by discussing the notation and results for Hall's basis theorem before proceeding to generalizations of this result.

Let $E$ be a free group on a free basis $x_1, \dots x_r$.  We define basic commutators and construct an ordering on the basic commutators inductively as follows:

\begin{itemize}
\item The basic commutators of weight 1 are the generators $x_1, \dots x_r$ with $x_i < x_j$ for $i<j$.

\item  The basic commutators of weight $n$ ate the commutators $c=[c_i,c_j]$ where $c_i$ and $c_j$ are basic commutators with weights summing to $n$.  These are ordered lexicographically: if $c=[c_i, c_j]$ and $c'=[c_i',c_j']$ then $c>c'$ if $c_i>c_i'$ or if $c_i=c_i'$ and $c_j>c_j'$. 
\end{itemize}

By imposing the additional requirement that $c_i<c_j$ if the weight of $c_i$ is less than the weight of $c_j$, we achieve an ordering of all basic commutators.

A \emph{formal commutator}, $c_i$ is an element of $E$ satisfying either of the following conditions:
\begin{itemize}
\item $c_k=x_i$ for some generator $x_i$
\item $c_k=[c_i, c_j]$ where $c_i$ and $c_j$ are formal commutators.
\end{itemize}
The formal commutators are weighted by the following rules:
\begin{itemize}
\item $\w(c_k)=1$ if $c_k=x_i$
\item $\w(c_k)=\w(c_i)+\w(c_j)$ if  $c_k=[c_i, c_j]$.
\end{itemize}

Note that weights induce a partial ordering on the formal commutators.  We say $c_i > c_j$ if $\w(c_i)>\w(c_j)$.  Given these weights, we define a special class of formal commutators known as basic commutators as follows.

\begin{definition}[Hall]
The element $c_k \in E$ is a basic commutator if it satisfies one of the following conditions:
\begin{itemize}
\item $c_k=x_i$ for some generator $x_i$
\item $c_k=[c_i, c_j]$ where $c_i$ and $c_j$ are basic commutators with $c_i>c_j$ and if $c_i=[c_s,c_t]$ then $\w(c_j)\ge \w(c_t)$.
\end{itemize}
\end{definition}

Given this restriction to basic commutators we may define a strict ordering on basic commutators.  We say $c_i < c_j$ if one of the following holds:

\begin{enumerate}
\item $c_i=x_i$ and $c_j=x_j$ for $i<j$.
\item $\w(c_i)<\w(c_j)$
\item $\w(c_i)=\w(c_j)$ and $c_i=[c_{i_1},c_{i_2}]$, $c_j=[c_{j_1},c_{j_2}]$ with $c_{j_1}<c_{j_1}$.
\item $\w(c_i)=\w(c_j)$ and $c_i=[c_k,c_s]$, $c_j=[c_k,c_t]$ with $c_s<c_t$.
\end{enumerate}

Above we have given a precise construction of a strict ordering on basic commutators.  While this ordering is consistent with Hall's original definition of ordering on basic commutators, he only insisted that the ordering be consistent with the partial ordering given by the weights and allowed for arbitrary ordering of commutators of the same weight.  For our generalizations of the basis theorem, we find it advantageous to work with the specific ordering given above.  This specific ordering of basic commutators has appeared before in \cite{W}.

To speak precisely about commutators and lower central series, we also introduce some new terminology.

\begin{definition}
Let $a_1, \dots, a_n \in G$.  We define an $n$-bracketing of $a_1, \dots, a_n$ inductively by

\begin{itemize}
\item The 1-bracketing of $a_1$ is $a_1$

\item A $n$-bracketing of $a_1, \dots, a_n$ is any commutator $[c_k,c_{n-k}]$ where $c_k$ is a $k$-bracketing of $a_1, \dots, a_k$ and $c_{n-k}$ is an $(n-k)$-bracketing of $a_{k+1}, \dots, a_n$.
\end{itemize}

We call an element $a_i$ in an $n$-bracketing of $a_1, \dots, a_n$ an entry.
\end{definition}

Note that the definition of $n$-bracketing is not very restrictive.  For example, the commutators  $[[a_1,a_2],[a_3,a_4]]$, $[[[a_1,a_2],a_3],a_4]$, and $[a_1,[a_2,[a_3,a_4]]]$ are all $4$-bracketings of the entries $a_1,a_2,a_3,a_4$.  Note that by the definition, any $n$-bracketing is an element of $G_n$, but it is possible for an $n$-bracketing to lie in a deeper term of the lower central series. 

Given these definitions for commutators, we are now equipped to approach the basis theorem.

\begin{theorem}[Hall]\label{T:basis}
If $E$ is the free group with free generators $x_1, \dots, x_r$ and if in a sequence of basic commutators $c_1, \dots, c_t$ are those of weight $1,2, \dots, k$ then an arbitrary element $g$ of $E$ has a unique representation 
$$
g=c_1^{e_1} \cdots c_t^{e_t} \,\, \mod \,\, E_{k+1}
$$
The basic commutators of weight $k$ form a basis for the free abelian group $E_k/E_{k+1}$.
\end{theorem}

Hall's basis theorem applies only to lower central series quotients of finitely generated free groups.  Below, we generalize the basis theorem to hold for lower central series quotients of infinitely generated free groups.

\begin{corollary}\label{C:basistheorem}
If $E$ is the free group with free generators $x_1, x_2, \dots$ and if in a sequence of basic commutators $\{c_i\}$ are those of weight $1,2, \dots, k$ then an arbitrary element $g$ of $E$ has a unique representation 
$$
g=\prod c_i^{e_i} \,\, \mod \,\, E_{k+1}
$$
where $e_i=0$ for all but finitely many $i$.  The basic commutators of weight $k$ form a basis for the free abelian group $E_k/E_{k+1}$.
\end{corollary}

\begin{proof}
Let $E(i)$ be the free group on $x_1, \dots, x_i$.  Note that the natural inclusion $E(i) \rightarrow E(j)$ for $j>i$ sends basic commutators to basic commutators and respects the ordering on basic commutators.

We first show that show that the map $\iota_{i,j}:E(i) \rightarrow E(j)$ induces an injection on the lower central series quotients 
$$
\overline{\iota}_{i,j}:E(i)_k/E(i)_{k+1} \hookrightarrow E(j)_k/E(j)_{k+1}.
$$
Note that $H_2(E(i), \Q)=H_2(E(i), \Q)=0$, as the wedge of $i$ or $j$ circles is a $K(G,1)$ for $E(i)$ or $E(j)$ respectively.  Furthermore, $\iota_{i,j}$ induces an injection $H_1(E(i), \Q) \rightarrow H_1(E(j), \Q)$.  By \prop{stallings}, $\iota_{i,j}$ induces an injection $\overline{\iota}_{i,j}:E(i)_k/E(i)_{k+1} \hookrightarrow E(j)_k/E(j)_{k+1}$ since free groups have the same rational lower central series and lower central series.

Let $\iota_i:E(i) \rightarrow E$ be the natural inclusion map sending $x_k \mapsto x_k$ \linebreak for $k \le i$.  By an analogous argument, $\iota_i$ induces an injection \linebreak $\overline{\iota}_{i}:E(i)_k/E(i)_{k+1} \hookrightarrow E_k/E_{k+1}$.

  Suppose $[x] \mapsto 0 \in E(i+1)_k/E(i+1)_{k+1}$.  Then there exists a representative $x$ of $[x]$ such that $x \in E(i+1)_{k+1}$.  However $x \in E(i)$ so by the basis theorem for finitely generated free groups, $x$ may be written uniquely in terms of basic commutators of $E(i)$: $x=c_1^{e_1} \cdots c_t^{e_t} \mod \,\, E(i)_{k+1}$.   This is also a representation for $x$ in terms of basic commutators of $E(i+1) \mod E(i+1)_{k+1}$, and hence the unique representation in $E(i+1)$.  However $x=0 \mod E(i+1)_{k+1}$, so because these basic commutators $c_1, \dots,c_t$ don't involve $x_{i+1}$, $x=0 \mod E(i)_{k+1}$. 

Given this it is easily checked that $\{E(i)_k/E(i)_{k+1}, \iota_{ij} \}$ is a directed system of groups.  We will show that $E_k/E_{k+1}$ is the direct limit of this \linebreak system. For this it suffices to show that for a group $G$ and maps \linebreak $f_i:E(i)_k/E(i)_{k+1} \rightarrow G$ such that $f_i=f_j \circ \iota_{ij}$, there exists a map \linebreak $f: E_k/E_{k+1} \rightarrow G$ such that $f \circ \iota_i=f_i$.  For any element $x \in E_k/E_{k+1}$, $x$ can be written as a finite length word in the generators $x_1, \dots$.  Hence $x \in E(i)_k/E(i)_{k+1}$ for some $i$.  Define $f(x)=f_i(x)$.  It is clear the resulting map $f$ is well defined and has the desired properties.  

To prove the first statement of the theorem, let $g \in E$, then $g \in E(i)$ for some $i$.  Hence in $E(i)$ there is a unique representation for $g$ as \linebreak $g=c_1^{e_1} \cdots c_t^{e_t} \,\, \mod \,\, E(i)_{k+1}$.  As $E(i)_k \subset E_k$, $g=c_1^{e_1} \cdots c_t^{e_t} \,\, \mod \,\, E_{k+1}$ is a representation of $g$ in the desired form in $E$.  Suppose there is another representation of this form, $g=d_1^{\epsilon_1} \cdots d_s^{\epsilon_s} \,\, \mod \,\, E_{k+1}$.  There exists a $j$ such that all of the basic commutators $c_1, \dots , c_t, d_1, \dots , d_s \in E(j)$.  Then by the basis theorem for finitely generated free groups these representations must be the same.

To prove the second statement, note that as $E_k/E_{k+1}$ is a direct limit of $E(i)_k/E(i)_{k+1}$, it follows from the basis theorem for finitely generated free groups that the basic commutators of weight $k$ generate $E_k/E_{k+1}$.  Furthermore, for any finite collection of basic commutators of weight $k$ $c_1, \dots ,c_m$, there is some $i$ such that $c_1, \dots , c_m \in E(i)_k/E(i)_{k+1}$.  Hence all commutators of weight $k$ are independent.  Therefore the basic commutators of weight $k$ form a basis for $E_k/E_{k+1}$.
\end{proof}

\begin{lemma}\label{L:homologybasistheorem}
Let $G$ be a group and let $a \in {G_k}$.  By the definition of the lower central series, $a$ is some $n$-bracketing of $a_{1}, \dots a_{n}$ where $a_{i} \in G$, $n \le k$, and $a_{i}\in G_{k_{i}}$ where $\sum_{i=1}^n k_i=k$.  Let $c_1, \dots, c_n$ be elements of $G$ with $c_i \in G_{k_{i}+1}$.  Let  $a'$ be the commutator $a$ where each entry $a_i$ is replaced by $c_ia_i$.  Then $a'=ca$ for some $c \in G_{k+1}$
\end{lemma}

\begin{proof}
We prove this using strong induction.  For the case $k=2$, \linebreak $a=[a_{i_1},a_{i_2}]$.  Using the commutator identities $[ux,y]= \lup{u}[x,y][u,y]$ and $[x,vy]=[x,v] \lup{v}[x,y]$ we can perform the following computation:

\begin{eqnarray*}
a'&=& [c_1 a_{i_1}, c_2 a_{i_2}]\\
&=& \lup{c_1}[a_{i_1}, c_2]  c_2 c_1 \lup{[a_{i_1}, a_{i_2}]} \left(c_1^{-1} c_2^{-1} [c_1, c_2 x_{i_2}]\right) [a_{i_1}, a_{i_2}]\\
\end{eqnarray*}

Hence for $c=\lup{c_1}[a_{i_1}, c_2]  c_2 c_1 \lup{[a_{i_1}, a_{i_2}]} \left(c_1^{-1} c_2^{-1} [c_1, c_2 a_{i_2}]\right)$ our base case holds.

For the inductive step, suppose $a=[a_m,a_n]$ where $a_m \in G_m$, and $a_n \in G_n$.  By the inductive hypothesis, $a_m'=c_m a_m$ where $c_m \in G_{m+1}$ and $a_n'=c_n a_n$ where $c_n \in G_{n+1}$

\begin{eqnarray*}
a'&=&[a_m',a_n']\\
&=& [c_m a_m, c_n a_n]\\
&=& \lup{c_m}[a_m, c_n]  c_n c_m \lup{[a_m, a_n]} \left(c_m^{-1} c_n^{-1} [c_m, c_n a_n]\right) [a_m, a_n]
\end{eqnarray*}

To finish the proof we must show 
$$
\lup{c_m}[a_m, c_n]  c_n c_m \lup{[a_m, a_n]} \left(c_m^{-1} c_n^{-1} [c_m, c_n a_n]\right) \in G_{n+m+1}.
$$ 
Note that $[a_m,c_n] \in G_{m+n+1}$, and so any conjugate is also in $G_{m+n+1}$.  We simplify the above expression as modulo $G_{m+n+1}$ as follows:

\begin{align*}
\lup{c_m}[a_m, c_n]  c_n c_m \lup{[a_m, a_n]} & \left(c_m^{-1} c_n^{-1} [c_m, c_n a_n]\right)\\
&=c_n c_m [a_m, a_n][c_m^{-1}, c_n^{-1}] [a_n c_m^{-1}] (c_n c_m)^{-1} [a_m, a_n]^{-1}
\end{align*}
As $[a_m,a_n] \in G_{m+n}$, it is in the center of $G/G_{m+n+1}$.  Also note that the commutators $[c_m^{-1},c_n^{-1}]$ and $[a_n, c_m^{-1}]$ are elements of $G_{n+m+1}$.  Thus the above expression is trivial modulo $G_{n+m+1}$.

Hence for $c=\lup{c_m}[a_m, c_n]  c_n c_m \lup{[a_m, a_n]} \left(c_m^{-1} c_n^{-1} [c_m, c_n a_n]\right)$ our induction holds.
\end{proof}

\begin{corollary}\label{P:commutatoridentity}
Commutators in $\frac{G_k}{G_{k+1}}$ are linear in each entry.  In other words, given $a,b, c \in G$.  If $C \in \frac{G_k}{G_{k+1}}$ be an $n$-bracketing with $[ab,c]$ as an entry.  Let $C'$ be the commutator obtained by replacing the entry $[ab,c]$ with $[a,c][b,c]$.  Then $C=C'$.  In addition, if $C \in \frac{G_k}{G_{k+1}}$ with an entry $[a,bc]$ and $C'$ is the commutator obtained by replacing the entry $[a,bc]$ with $[a,b][a,c]$, then $C=C'$.
\end{corollary}
\begin{proof}
We will prove the first identity.  The proof of the second is analogous.

In any group we have the identity 
\begin{align*}
[ab,c]&=\lup{a}[b,c][a,c]\\
&=[a,c]\lup{[c,a]a}[b,c]\\
&=[a,c][\lup{[c,a]a}b,\lup{[c,a]a}c].
\end{align*}

Let $b \in G_{k_b}$ and $c \in G_{k_c}$.  Note that the elements $b$ and $\lup{[c,a]a}b$ share the same class in $G/G_{k_b+1}$.  Similarly, the elements $c$ and $\lup{[c,a]a}c$ share the same class in $G/G_{k_c+1}$.  The result follows immediately from \lem{homologybasistheorem}.
\end{proof}

%
%

The following proposition establishes a relationship between bases for the lower central series quotients $\frac{E(n)_k}{E(n)_{k+1}}$ and $\frac{E(n-1)_k}{E(n-1)_{k+1}}$.

\begin{proposition}\label{P:splitting}
Let $E(n-1)$ be the free group on $\{x_1, \dots, x_{n-1}\}$ and let $E(n)$ be the free group on $\{x_1, \dots, x_{n}\}$.  Let 
$$
\pi: E(n) \rightarrow E(n-1) \cong E(n)/\langle x_n \rangle
$$
 be the natural quotient map.  The kernel of the induced map
$$
\overline{\pi}: \frac{E(n)_k}{E(n)_{k+1}} \rightarrow \frac{E(n-1)_k}{E(n-1)_{k+1}}
$$
is generated by weight $k$ basic commutators which have $x_n$ as an entry.
\end{proposition}

\begin{proof}
Let $K$ be the subgroup of $\frac{E(n)_k}{E(n)_{k+1}}$  generated by basic commutators which have $x_n$ as an entry.  We show $K=\ker \overline{\pi}$.  First, consider a basic commutator $c$ of weight $k$ which has $x_n$ as an entry.  We show that $\pi(c)=1$ using strong induction.

In the first case we consider a weight 2 basic commutator.  Suppose $c=[x_i,x_n]$ or $c=[x_n,x_i]$ for some $x_i$.  Then \begin{align*}
{\pi}(c)&=[\pi(x_i),\pi(x_n)]\\
&=[x_i,\pi(x_n)]\\
&=1.
\end{align*}
and similarly for $c=[x_n,x_i]$

Suppose for induction that for all commutators $c$ of weight less than $k$, that $\pi(c)=1$.  Let $c$ be a weight $k$ basic commutator with $x_n$ as an entry.  Then $c=[c_1,c_2]$ where $c_1$ and $c_2$ are basic commutators of weight $\le k-1$.  Either $c_1$ or $c_2$ must have $x_n$ as an entry.  If $c_1$ has $x_n$ as an entry then
\begin{align*}
\pi(c)&=[\pi(c_1),\pi(c_2)]\\
&=[1, \pi(c_2)]\\
&=1
\end{align*}
and similarly if $c_2$ has $x_n$ as an entry.  Thus $K \subset \ker \overline{\pi}$.

To show the opposite inclusion, let $\iota: E(n-1) \rightarrow E(n)$ be the natural inclusion map.  This map induces a monomorphism $\overline{\iota}: \frac{E(n-1)_k}{E(n-1)_{k+1}} {\rightarrow} \frac{E(n)_k}{E(n)_{k+1}}$.  Furthermore as $\overline{\pi}\circ \overline{\iota}$ is the identity map, $\overline{\pi}$ is a retract.
$$
 \frac{E(n-1)_k}{E(n-1)_{k+1}} \stackrel{\overline{\iota}}{\rightarrow} \frac{E(n)_k}{E(n)_{k+1}} \stackrel{\overline{\pi}}{\rightarrow} \frac{E(n-1)_k}{E(n-1)_{k+1}}
$$
Suppose $c \in  \frac{E(n)_k}{E(n)_{k+1}}$ and $c \notin K$.  Then by the basis theorem $c$ can be written as a product of  basic commutators $c_1, \dots, c_m$ in the generators $\{x_1, \dots, x_{n}\}$.  In this product there is some set of basic commutators $\{c_{i}|i \in A\}$ for which $c_{i} \notin K$.  Then by definition of $K$, for each $i \in A$, $c_{i}$ is a basic commutator in the generators $\{x_1, \dots, x_{n-1}\}$.  Thus for each $i \in A$, $c_{i} \in \im \overline{\iota}$ and the product of these basic commutators $\prod_{i \in A} c_{i}$ is nontrivial.  Then as $\overline{\pi}(c_{i})=c_i$ for each $i \in A$ and $\overline{\pi}(c_i)=1$ for $i \notin A$, it follows that $\overline{\pi}(c)=\prod_{i \in A} c_{i}$.  Hence $c \notin \ker \overline{\pi}$.  Therefore, $K=\ker \overline{\pi}$, as desired.
\end{proof}

\subsection{Module structure of the lower central series quotients}

If $E$ is a free group on $\{x_1, \dots, x_n\}$ and $E'$ is  its commutator subgroup, we may employ a particular basis for $E'$, given in \cite{T}.  We reproduce the basis here in a slightly modified form to suit our purposes.

\begin{theorem}(Tomaszewski)\label{T:Basis}
$E'$ is freely generated by the set
$$
B=\{[x_i, x_k]^{x_1^{d_1} \cdots x_k^{d_k}}|1\le i < k \le n, d_j \in \Z\}.
$$
\end{theorem}

We use this free basis for $E'$ in conjunction with the tools developed in \sec{basistheorems} to prove the following lemma.

\begin{lemma}\label{L:torsion}
The module $\frac{E'_k}{E'_{k+1}}$ has no $\Z \left[ \frac{E}{E'} \right]$ torsion of the form \linebreak $(1-x_i)\omega=0$, where $x_i$ is a generator for $E$.
\end{lemma}

\begin{proof}
Consider the free basis $B$ for $E'$ given by \thm{Basis}.  Note that the set of basis elements, $B$, maps to itself under conjugation by $x_1$.  Given an element $\omega \in \frac{E'_k}{E'_{k+1}}$, by \cor{basistheorem}, $\omega$ can be written as a product $\omega= \prod_{i=1}^m c_i^{\alpha_i}$ where $c_i$ are basic commutators in the elements of $H_1(F')$ and $c_i < c_{i+1}$ for all $i$.

Let $X$ be the wedge of $n$ circles and let $\widetilde X$ denote the universal \linebreak abelian cover of $X$.  Note that $\pi_1(X)=E$.  Thus $\pi_1\left(\widetilde X\right)=E'$ and \linebreak $H_1\left(\widetilde X\right)=H_1(E')$.  Let $v$ be the vertex of $X$.  Then by the long exact sequence of a pair we have that $i:H_1(E') \hookrightarrow H_1\left(\widetilde X, \tilde v \right)$, where the $i$ is induced by the natural inclusion.  Furthermore, as $B$ is a basis for $E'$, the induced function $i:B \rightarrow H_1\left(\widetilde X, \tilde v \right)$ is injective.  The universal abelian cover, $\tilde{X}$, is a 1-complex taking the form of a square grid with a dimension corresponding to each generator.  The vertices of this grid can be labeled by a vector $(a_1, a_2, \dots, a_n)$ where $a_i$ denotes the distance of the vertex in the $x_i$ direction.  These vertices can be ordered by the dictionary order.  That is, if $v=(a_1, \dots, a_n)$ and $v'=(a_1', \dots, a_n')$ then $v<v'$ if $a_1<a_1'$ or $a_i=a_i'$ for all $i<j$ and $a_j<a_j'$.

Any edge in the lattice can then be written as an ordered pair of vertices, $e=(v_1,v_2)$ with $v_1<v_2$.  The edges then inherit a strict ordering by the dictionary order on the weighted vertices.  That is, if $e=(v_1,v_2)$ and $e'=(v_1',v_2')$ then $e<e'$ if $v_1<v_1'$ or if $v_1=v_1'$ and $v_2<v_2'$. 

Any homology element of $H_1\left(\widetilde X, \tilde v \right)$ can be written as a finite sum of edges, $c=\sum_{i=1}^{l} b_i e_i$ where $e_i \le e_{i+1}$.  Thus the weighting on oriented edges of $E'$ induces an ordering on $H_1\left(\widetilde X, \tilde v \right)$ by the dictionary order in the same way.  We may represent any element $c$ as a vector $(b_1, b_2, \cdots)$ where $b_i$ is the coefficient of the edge $e_i$.  Note that by construction, such a vector has finitely many nonzero entries.   For $c=(b_1, \dots)$ and $c'=(b_1', \dots)$ then $c<c'$ if $b_1<b_1'$ or $b_i=b_i'$ for all $i<j$ and $b_j<b_j'$.  In this manner we obtain a strict ordering on $H_1\left(\widetilde X, \tilde v \right)$ and hence also on the basis $B$.  We can use this ordering of our basis to construct our basic commutators as in the basis theorem.  Conjugating by $x_1$ preserves the ordering on $H_1\left(\widetilde X, \tilde v \right)$ as it preserves the ordering on vertices.  Hence since $B$ is invariant under conjugation by $x_1$, each basic commutator $c_i$ in the product $\omega= \prod_{i=1}^m c_i^{\alpha_i}$ where $c_i$ is sent to another basic commutator in the elements of $B$, $c_i'$.  As the ordering on elements of $B$ is preserved under the conjugation by $x_1$, the $c_i'$ also satisfy $c_i'<c_{i+1}'$.

Hence by \cor{basistheorem}, $x_1 \cdot \omega$ can be written uniquely by the basis theorem as  $x_1 \cdot \omega= \prod_{i=1}^m c_i'^{\alpha_i}$, where $c_i'=\lup{x_1}c_i$. Note that $\omega \ne x_1 \cdot \omega$ as $c_1 \ne c_1'$ and thus the unique expressions of $\omega$ and $\lup{x_1} \omega$ have distinct least commutators.  Thus $(1-x_1)\omega \ne 0$.

The result that $(1-x_i)\omega \ne 0$ for any $i$ can be obtained by reordering the generators of $E$ such that $x_i$ plays the role of $x_1$.
\end{proof}

Note that $(1-x_1)(1-x_2) \omega \ne 0$ is equivalent to the statement that the map $\cdot (1-x_1)(1-x_2): \frac{E'_k}{E'_{k+1}} \rightarrow \frac{E'_k}{E'_{k+1}}$ is injective.  Hence if $\omega \ne \omega'$ then $(1-x_1)(1-x_2) \omega \ne (1-x_1)(1-x_2) \omega'$.    This fact will be employed in future Magnus homomorphism computations.

\section{Main Results}
In this Section we investigate properties of the higher-order Magnus subgroups.  \sec{subsurfaces} develops a way of obtaining mapping classes in $M_k(S)$ from those in $J_k(D)$ and shows these mapping classes to be nontrivial.  Given that it is known that $J_k(D)$ is nontrivial for all $k$, this shows that the higher dimensional analog $M_k(S)$ is nontrivial for all $k$ for genus $\ge 3$.  In \sec{maghomom} we seek to strengthen this result.  We will show that the Magnus homomorphisms are nontrivial on $M_k(S)$ given some conditions on $\tau_k(J_k(D))$.  Using these Magnus homomorphism computations we will exhibit a subgroup of $\frac{M_k}{M_{k+1}}$ isomorphic to a lower central series quotient of free groups.  Finally, in \sec{mainthm} we will show that $\frac{M_k(S)}{ M_{k+1}(S)}$ is infinite rank for all $k$.

\subsection{Constructing elements of $M_k$ via subsurfaces}\label{S:subsurfaces}

Let $S$ be a surface with genus $g \ge 2$ and 1 boundary component.  Let $D$ be a sphere with $n$ disks removed, $n \ge 3$.  We work to relate Johnson filtration on $D$ to the Magnus filtration on $S$ by considering separating embeddings of $D$ in $S$.

\begin{definition}
Let $D$ have boundary components $b_0, \dots, b_{n}$.  The map $i:D \rightarrow S$ is a separating embedding if $i$ is an embedding such that $i(b_1), \dots, i(b_n)$ are separating curves in $S$ and $i(b_0)$ is either a separating curve in $S$ or is the boundary component of $S$.
\end{definition}

\begin{figure}[h]
\begin{center}
\includegraphics[width=5in]{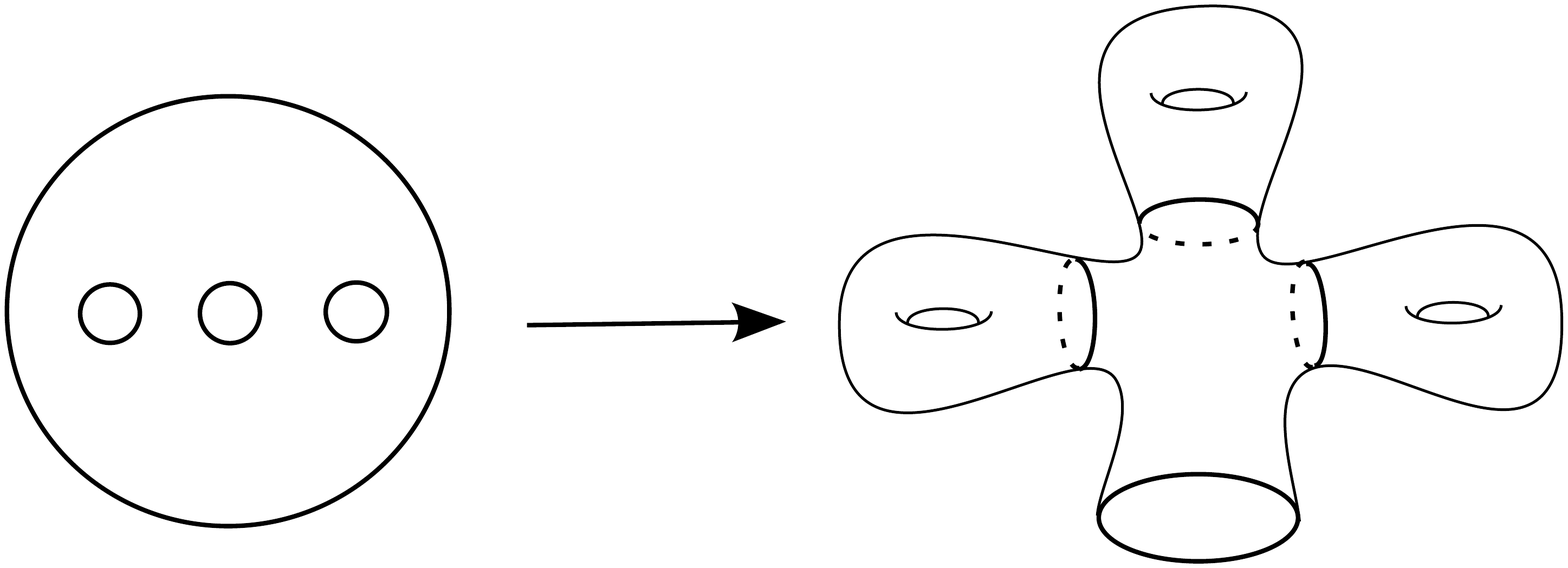}
\put(-297,182){\scriptsize $\dots$}
\put(-68,200){$\ddots$}
\vspace{-4cm}\caption[A separating embedding of $D$ in $S$.]{To obtain examples of $f' \in M_k(S)$ from $f \in J_k(D)$, we embed the disk, $D$ into $S$ such that each boundary component of $D$ is either separating in $S$ or is the boundary component of $S$.  The above illustrates a possible separating embedding of $D_g$ in $S_g$.}
\label{F:embedding}
\end{center}
\end{figure}

We first develop a relationship between the Johnson subgroups on a disk and the Magnus subgroups on a larger surface.  For this we will employ a specific basis for $F$ that is compatible with the arcs which generate $H_1(D, \bd D)$.   Let $*$ be a basepoint for $F=\pi_1(S)$ located on the boundary of $S$.  Let $A_i$ be arcs connecting the $i^{th}$ boundary component to $p_0$ as in \defn{multipleboundary}.  Let $p_i$ be the terminal point of $A_i$.  As the boundary components of $i(D)$ are separating in $S$, $S \setminus D$ is a disjoint union of at most $n+2$ surfaces, one of which is $i(D)$.  Let us denote the other surfaces $\Sigma_0, \dots, \Sigma_n$, with $\Sigma_0$ chosen such that $\Sigma_0$ contains the boundary component of $S$ (note that if $i$ maps a boundary component of $D$ to the boundary component of $S$, $\Sigma_0$ is empty).  Let $\Sigma_i$ have genus $g_i$.  Then $\pi_1(\Sigma_i, p_i)$ has a basis consisting of $2g_i$ loops (given the extra boundary component, $\Sigma_0$ will have a basis of $2g_0+1$ loops, but we will only consider the $2g_0$ loops which form a basis for the capped off surface).  By the Seifert Van Kampen theorem, we can combine these bases to form a basis for $F$ as follows:  Let $C$ be an arc joining $*$ to $p_0$.  The elements of our basis for $S$ are the homotopy classes of the loops $CA_i \beta \overline{A_i}\,\overline{C}$ (or $C\beta \overline{C}$ for $i=0$) where $\beta$ is a generator of $\pi_1(\Sigma_i, p_i)$.  This basis is illustrated in \figr{embeddingbasis} below.  We denote the elements of this basis $\{\alpha_1, \gamma_1, \dots, \alpha_g, \gamma_g\}$ where the curves $\alpha_{g_0+ \cdots+g_{i-1}+1}, \gamma_{g_0+ \cdots +g_{i-1}+1}, \dots, \alpha_{g_0+ \cdots + g_i}, \gamma_{g_0+ \cdots + g_i}$ are those basis elements produced by the generators of $\pi_1(\Sigma_i, p_i)$.


\begin{figure}[h] 
\centerline{\includegraphics[width=6.9in]{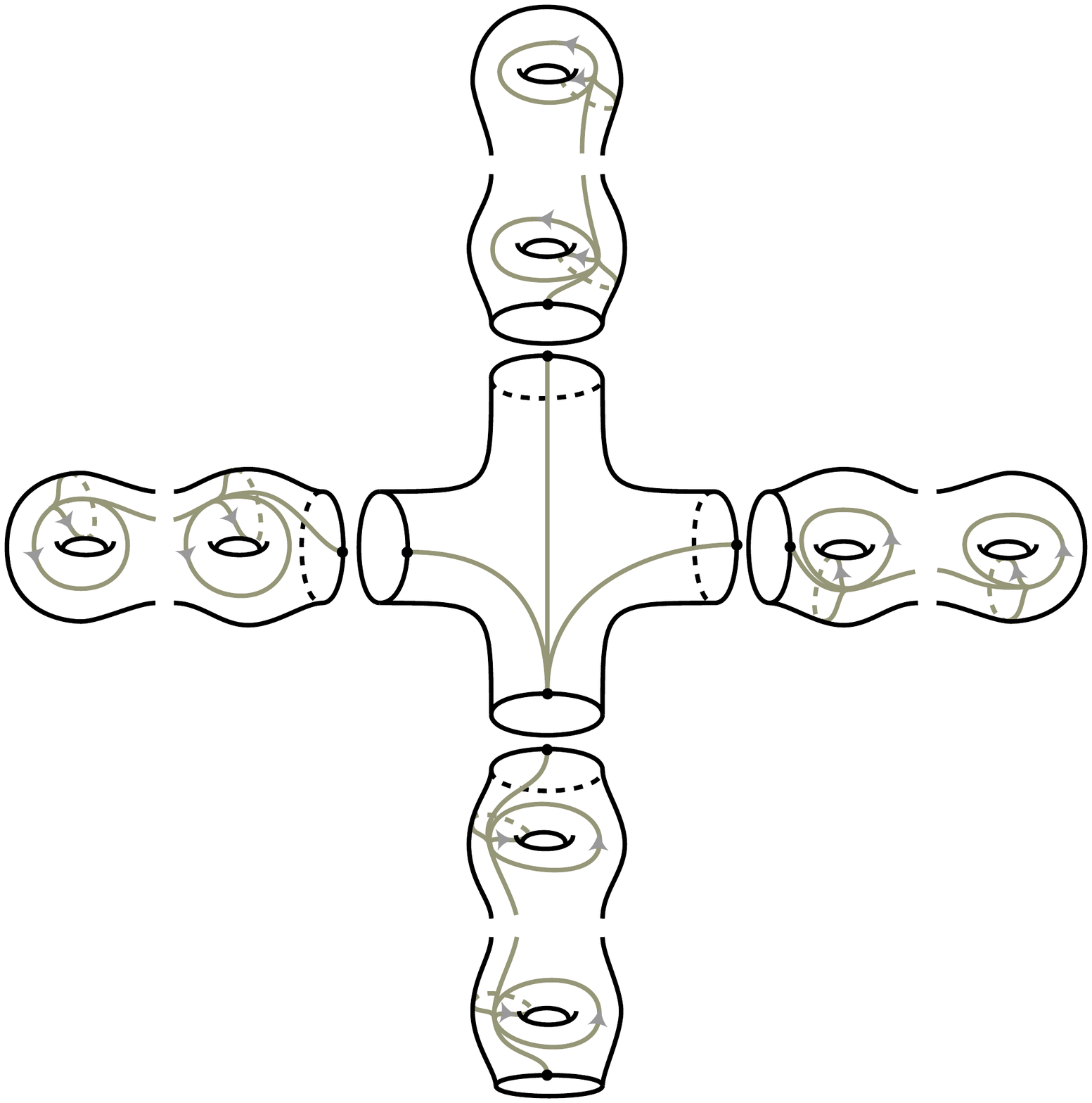} 
		\put(-242,13){$*$}
	\put(-243,105){$p_0$}
	\put(-243,123){$p_0$}
         	 	\put(-298,170){$p_1$}
         		\put(-318,170){$p_1$}
	\put(-255,245){$p_2$}
	\put(-255,228){$p_2$}
		\put(-178,180){$p_n$}
		\put(-196,181){$p_n$}
	\put(-243,55){$\vdots$}
	\put(-370,175){$\dots$}
	\put(-245,295){$\vdots$}
	\put(-130,175){$\dots$}
	\put(-210,210){$\ddots$}
		\put(-240,45){$\alpha_{g_0}$}
		\put(-240,68){$\alpha_{1}$}
	\put(-280,88){$\gamma_1$}
	\put(-282,32){$\gamma_{g_0}$}
		\put(-350,206){$\gamma_{g_0+1}$}
		\put(-400,206){$\gamma_{g_0+g_1}$}
	\put(-350,145){$\alpha_{g_0+1}$}
	\put(-400,145){$\alpha_{g_0+g_1}$}
		\put(-218, 258){$\gamma_{g_0+g_1+1}$}
		\put(-218,312){$\gamma_{g_0+g_1+g_2}$}
	\put(-315,270){$\alpha_{g_0+g_1+1}$}
	\put(-315,325){$\alpha_{g_0+g_1+g_2}$}
		\put(-175,145){$\gamma_{g_0+ \cdots+g_{n-1}+1}$}
		\put(-100,145){$\gamma_{g}$}
	\put(-175,205){$\alpha_{g_0+ \cdots+g_{n-1}+1}$}
	\put(-100,205){$\alpha_g$}
		\put(-240,190){$A_2$}
		\put(-280,177){$A_1$}
		\put(-220,177){$A_n$}
	\put(-252,55){$C$}}
   \caption[The chosen basis $\{\alpha_1,\gamma_1, \dots, \alpha_g, \gamma_g \}$ of $F$]{Pictured above is the chosen basis $\{\alpha_1,\gamma_1, \dots, \alpha_g, \gamma_g \}$ of $F$, obtained by connecting the bases for $\pi_1(\Sigma_i, p_i)$ to the basepoint $*$ via the arcs $A_i$.}
   \label{F:embeddingbasis}
\end{figure}

\begin{lemma}\label{L:inclusion}
Let $i:D \rightarrow S$ be a separating embedding.  Let $[f]\in Mod(D)$ and let $f$ be a representative homeomorphism of $[f]$.  Let $f':S \rightarrow S$ be the homeomorphism defined by
$$
f'(x)= \left\{ \begin{array}{ll} i(f(y)) & x=i(y)\\ x & x \in S \setminus i(D) \end{array} \right.
$$
then if $[f] \in J_k(D)$, $[f'] \in M_k(S)$.
\end{lemma}

\begin{proof}
Choose an ordering of the boundary components of $D$, points $p_i$ on these boundary components and arcs $A_i$ as in \defn{multipleboundary} such that the boundary component of $S$ is contained in the component of $S \setminus \text{int}\, i(D)$ containing the $0^{th}$ boundary component of $D$.  Let $*$ be a basepoint for $\pi_1(S)$ which lies on $\partial S$ and let $c$ be an arc parametrized on $[0,1]$ such that $c(0)=*$ and $c(1)=p_0 \in \partial D$.  By construction of our basis for $\pi_1(S, *)$ in which each generator can be represented by a loop $\alpha$ which is either disjoint from $i(D)$, or is of the form $\alpha= C A_i \beta \overline{A_i}\,\overline{C}$ with $\beta$ a loop intersecting $i(D)$ only at its initial and terminal points.  

For $\alpha$ disjoint from $i(D)$, $f'_*(\alpha)=\alpha$ and thus $f'_*(\alpha) \alpha^{-1}=1$ is trivially contained in $F_k$.

For $\alpha=C A_i \beta \overline{A_i}\,\overline{C}$ we can perform the following computation.

\begin{align*}
f'(\alpha)\overline{\alpha}&\simeq f'(C A_i \beta \overline{A_i}\,\overline{C})\overline{(C A_i \beta \overline{A_i}\,\overline{C})}\\
&\simeq f'(C) f'(A_i) f'(\beta) f'(\overline{A_i}) f'(\overline{C}) C A_i \overline{\beta} \, \overline{A_i}\,\overline{C}\\
&\simeq C\, i(f(A_i)) \beta \, i(f(\overline{A_i})) \overline{C} C A_i\overline{\beta} \, \overline{A_i}\,\overline{C}\\
&\simeq \left(C\, i(f(A_i))\overline{A_i}\,\overline{C}\right) \left(C A_i \beta \overline{A_i}\,\overline{C}\right) \left(C A_i i(f(\overline{A_i})) \overline{C}\right) \left( C \overline{\beta} \, \overline{A_i}\,\overline{C}\right) \\
&\simeq i_*\left(f(A_i)\overline{A_i}\right) \left(C A_i \beta \overline{A_i}\,\overline{C}\right) i_*\left( A_i f(\overline{A_i})\right) \left( C \overline{\beta} \, \overline{A_i}\,\overline{C}\right) 
\end{align*}
Note that $\left(f(A_i)\overline{A_i}\right)^{-1}=A_i f(\overline{A_i})$.  As $f \in J_k(D)$, both $A_i f(\overline{A_i})$ and $f({A_i})\overline{A_i}$ are contained in $\pi_1(D)_k$.  Each boundary curve of $i(D)$ is the boundary of a subsurface of $S$ and hence is contained in $[F,F]$.  Since  $\pi_1(D)$ is generated by the boundary curves of $D$, it follows that $i_*(\pi_1(D)) \subset F'$, and hence $i_*(\pi_1(D)_k) \subset F'_k$.  Hence both $i_*(A_i f(\overline{A_i}))$ and $i_*(f({A_i})\overline{A_i})$ are contained in $F'_k$.  Considering the above expression modulo $F'_k$ we then achieve the following.

\begin{align*}
f'(\alpha)\alpha^{-1}&=CA_i \beta \overline{A_i}\,\overline{C} CA_i \overline{\beta} \overline{A_i} \, \overline{C}&\mod F'_k\\
&=\alpha \alpha^{-1}=1 & \mod F'_{k}
\end{align*}
Therefore $f' \in M_k(S)$.

\end{proof}

\lem{inclusion} allows us to construct numerous examples of elements of $M_k(S)$ by extending homeomorphisms in $J_k$ of embedded disks.

\begin{proposition}\label{P:homomorphism}
Let $i:D \rightarrow S$ be a separating embedding.  The map $i':Mod(D) \rightarrow Mod(S)$ given by $i'([f])=[f']$ is an injective homomorphism.  This map induces a homomorphism $$\overline{i}:\frac{J_k(D)}{J_{k+1}(D)} \rightarrow \frac{M_k(S)}{M_{k+1}(S)}.$$
\end{proposition}

\begin{proof}
We begin by showing that $i'$ is multiplicative.  Consider maps \linebreak $[f_1], [f_2] \in Mod(D)$ and let $f_1,f_2$ be corresponding homeomorphisms.  Clearly as elements of $Mod(D)$, $[f_1][f_2]=[f_1f_2]$.  The composition $f_1 f_2$ is a representative of the class $[f_1f_2]$.  We then have:
\begin{align*}
i'\left([f_1][f_2]\right)&=i' \left(\left[f_1f_2\right]\right)\\
&=\left[(f_1f_2)'\right].\\
\end{align*}
Note that by definition $(f_1f_2)'$ is the homeomorphism $S \rightarrow S$ which extends $f_1f_2$ by the identity.  We then have that $(f_1f_2)'=f_1'f_2'$.  By definition of multiplication in $Mod(S)$, $\left[f_1'f_2'\right]=[f_1'][f_2']$.  Thus,
\begin{align*}
i'\left([f_1][f_2]\right)&=\left[(f_1f_2)'\right]\\
&=\left[f_1'f_2'\right]\\
&=[f_1'][f_2']\\
&=i'\left([f_1]\right)i'\left([f_2]\right)
\end{align*}
Thus $i'$ is multiplicative.

To show that $i'$ is injective, it then suffices to show that $\ker{i'}=1$.  This amounts to showing that beginning with a nontrivial mapping class \linebreak $f \in M_k(D)$, the resulting mapping class $f' \in M_k(S)$ is necessarily nontrivial.  As no boundary component of $D$ is nullhomotopic in $S$, this follows directly from \cite{FM} Theorem 3.18.  Hence $i'$ is a monomorphism.

We now address the second part of the proposition: that $i'$ induces a homomorphism $\overline{i}:\frac{J_k(D)}{J_{k+1}(D)} \rightarrow \frac{M_k(S)}{M_{k+1}(S)}$.  By \lem{inclusion} $i'(J_{k+1}(D)) \subset M_{k+1}(S)$.  Thus the map $\overline{i}$ is well defined.  It is clearly a homomorphism as $i'$ is a homomorphism.  This completes the proof.
\end{proof}

\subsection{Magnus homomorphism computations}\label{S:maghomom}

Having developed a relationship between Johnson subgroups on $D$ and Magnus subgroups on $S$, we now seek to relate the Johnson homomorphisms on $D$ to the Magnus homomorphisms on $S$.  To do this we must first examine the relationship between the lower central series quotients of $\pi_1(D)$ and $F'$.

 Let $G$ denote the fundamental group of $D$, the disk with $n$ holes, and let $y_i$ be the generators of $G$ obtained by traveling along arc $A_i$, circling the corresponding boundary component in a counterclockwise direction, and returning to the basepoint along $\overline{A_i}$ as shown in \figr{arcs}.  
\begin{figure}[h] 
   \centering
   \includegraphics[width=3in]{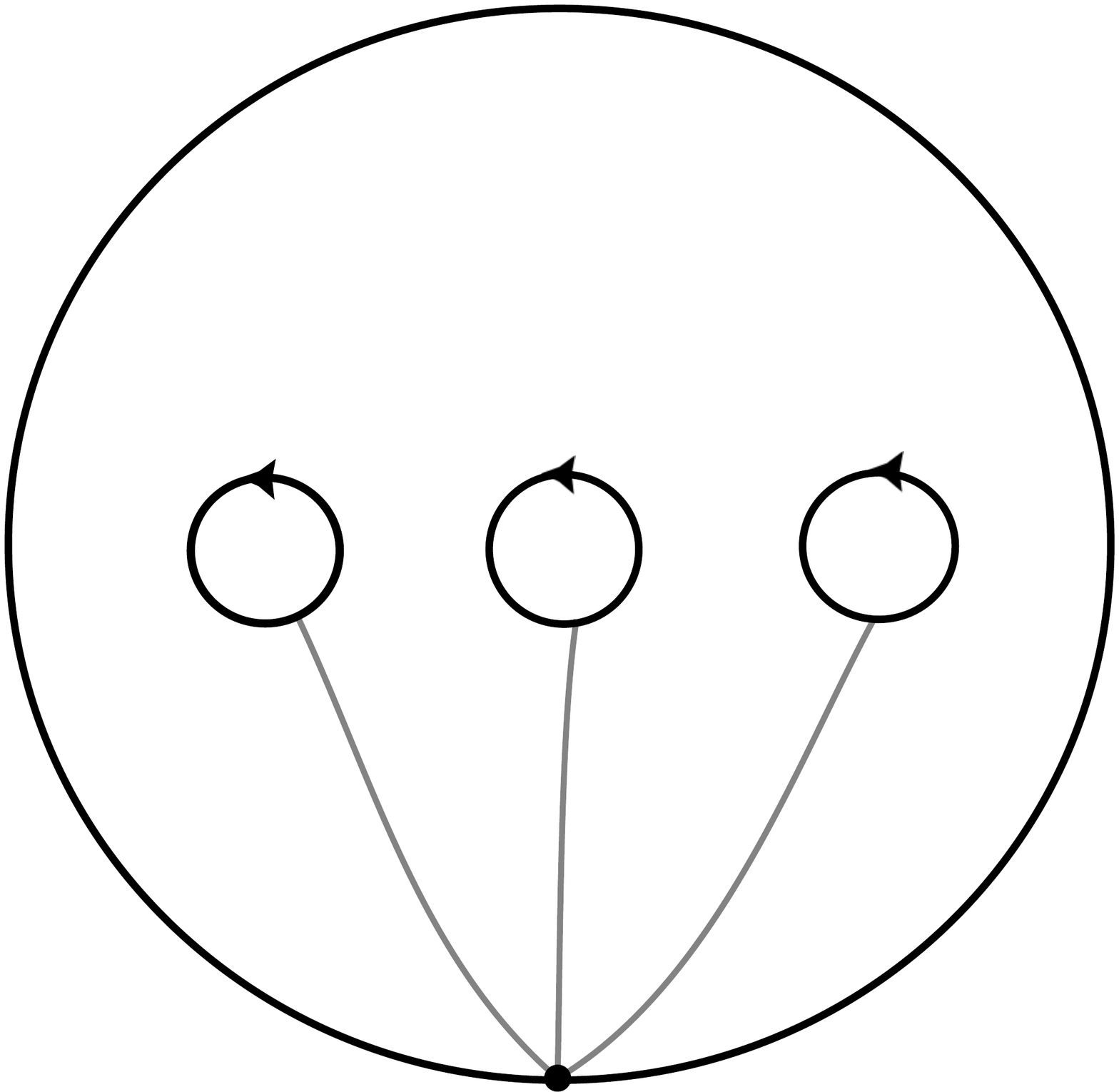} 
   \put(-80,38){$A_n$}
      \put(-75,90){$y_n$}
         \put(-95,75){$\dots$}
         \put(-110,38){$A_2$}
            \put(-150,38){$A_1$}
               \put(-160,90){$y_1$}
                  \put(-120,90){$y_2$}
   \caption[The generators $y_i$ of $G=\pi_1(D)$.]{Pictured above are the generators $y_i$ of $G$.  A generator $y_i$ is obtained by traveling along arc $A_i$, circling the corresponding boundary component in a counterclockwise direction, and returning to the basepoint along $\overline{A_i}$.}
   \label{F:arcs}
\end{figure}

\begin{lemma}\label{L:braidsinject}
Let $i:D \rightarrow S$ be a separating embedding.  The induced map  $i_*: \frac{G_k}{G_{k+1}} \rightarrow \frac{F'_k}{F'_{k+1}}$ is injective.
\end{lemma}

\begin{proof}
To show that the above map is an injection, we will employ \prop{stallings}.   Hence we must show that the homomorphism $i_*:G \rightarrow F'$ given by $y_i \mapsto [\alpha_i, \gamma_i]$ induces an injection $H_1(G;\mathbb{Q}) \rightarrow H_1(F'; \mathbb{Q})$.  As $G/G'$ and $F'/F''$ are torsion free abelian groups, it suffices to show there is an injection $H_1(G;\mathbb{Z}) \rightarrow H_1(F'; \mathbb{Z})$.  Note that by our previous construction of the basis for $F'$, 
$$
i_*(y_i)= \left[\alpha_{g_0+ \cdots+g_{i-1}+1}, \gamma_{g_0+ \cdots +g_{i-1}+1}\right] \cdots \left[\alpha_{g_0+ \cdots + g_i}, \gamma_{g_0+ \cdots + g_i}\right].
$$
Consider an element $\sum n_i y_i$ which is nonzero in $G/G'$.  We compute the image of this element by $i_*$ as follows:
\begin{align*}
i_*\left(\sum n_i y_i\right)&=\sum n_i i_*(y_i)\\
&= \sum n_i \left[\alpha_{g_0+ \cdots+g_{i-1}+1}, \gamma_{g_0+ \cdots +g_{i-1}+1}\right] \cdots \left[\alpha_{g_0+ \cdots + g_i}, \gamma_{g_0+ \cdots + g_i}\right]
\end{align*}
Because $\sum n_i y_i \ne 0$, $n_j \ne 0$ for some $j$.  Consider the map $g_j$ which maps $S$ to the punctured surface as shown in \figr{torus}.  

\begin{figure}[h] 
   \centering
   \includegraphics[width=3in]{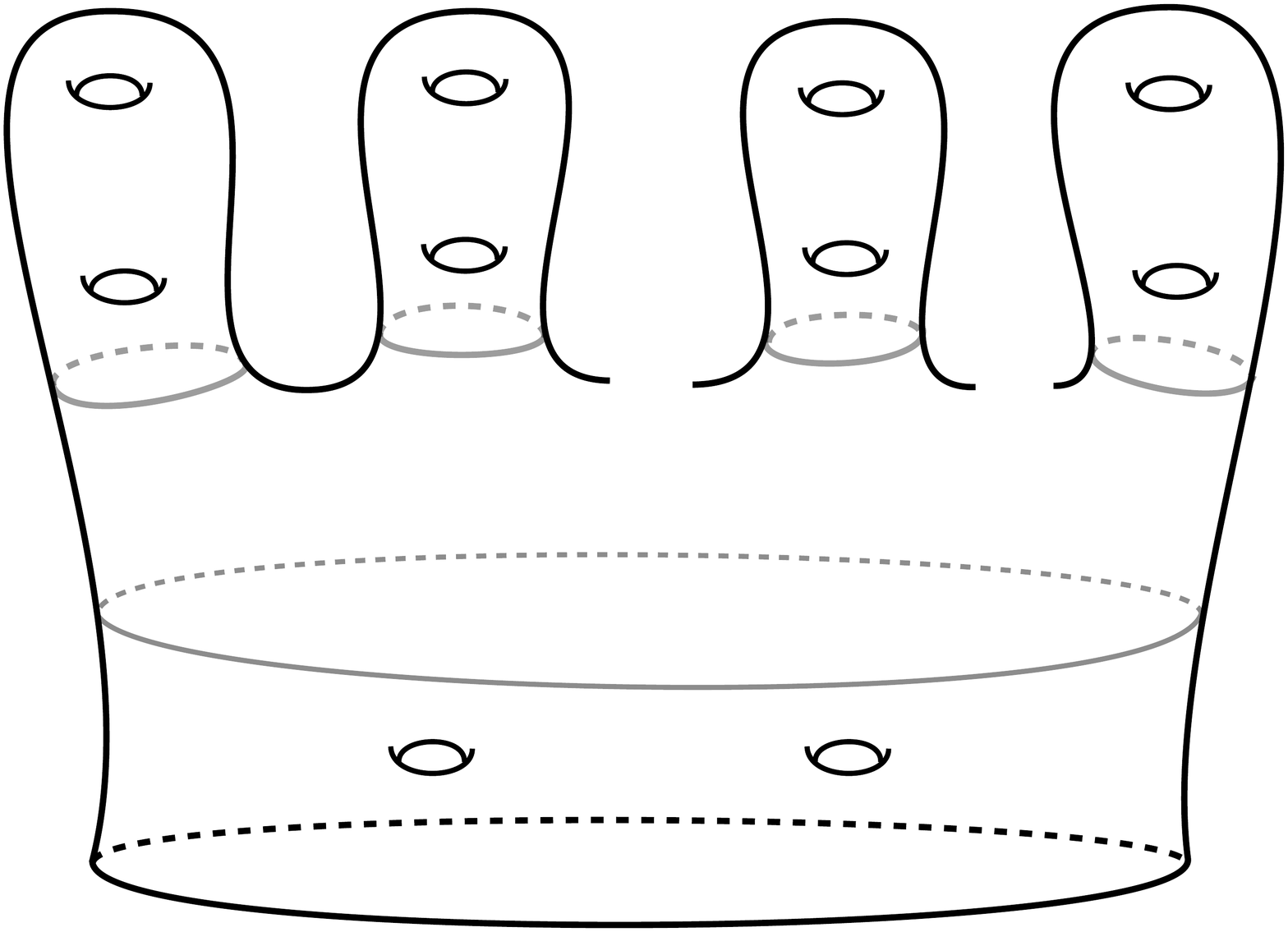}
   \put(-38,77){$i(b_n)$}
\put(-117,48){$i(b_0)$}
   \put(-197,75){$i(b_1)$}
   \put(-149,82){$i(b_2)$}
   \put(-87,82){$i(b_j)$}
    \put(-115,105){$\dots$}
    \put(-55,105){$\dots$}
\put(-197,117){$\vdots$}
\put(-140,119){$\vdots$}
\put(-78,118){$\vdots$}
\put(-24,117){$\vdots$}
\put(-117,27){$\dots$}

\vspace{.5cm}
 
 \includegraphics[width=.1in]{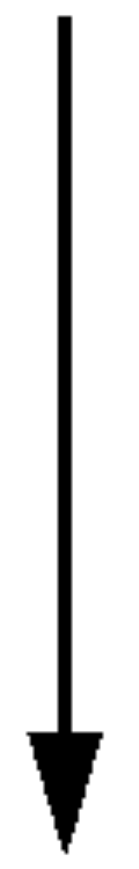}
   \put(0,25){$g_j$}
   
   \vspace{.5cm}
   \includegraphics[width=3in]{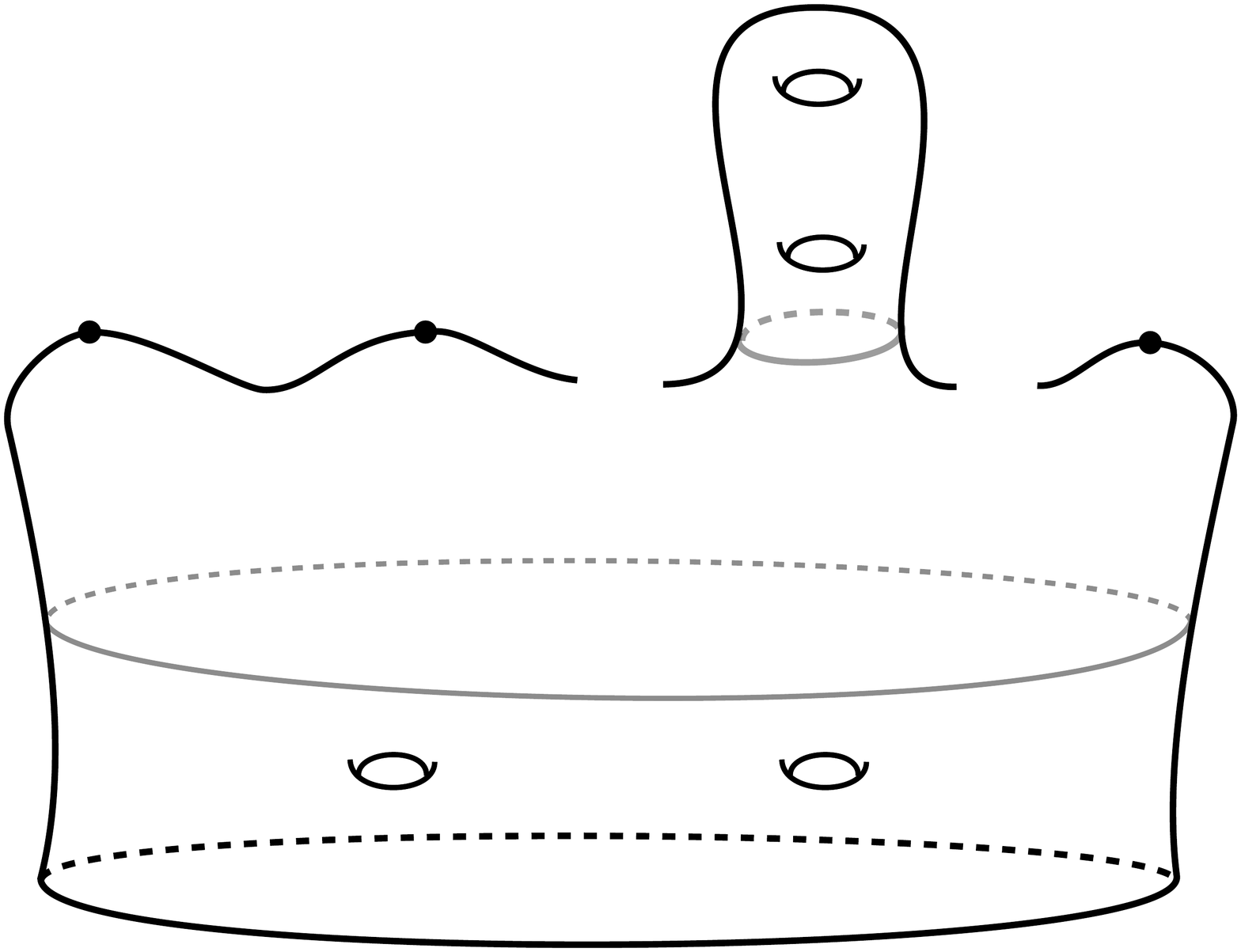} 
                  \put(-200,103){$p_1$}
                  \put(-142,103){$p_2$}
                  \put(-30,102){$p_n$}
                  \put(-115,95){$\dots$}
\put(-78,118){$\vdots$}
    \put(-55,95){$\dots$}
\put(-117,27){$\dots$}
   \caption[The continuous map $g_j:S \rightarrow T$]{Pictured above is the continuous map $g_j:S \rightarrow T$.  Everything above and including the curve $i(b_i)$ for $i \ne j$ is collapsed to a point $p_i$.}
   \label{F:torus}
\end{figure}

We find that 
\begin{align*}
g_{j*}i_*&\left(\sum n_i y_i\right)=\sum n_i i_*(y_i)\\
&= g_{j*}\left(\sum n_i \left[\alpha_{g_0+ \cdots+g_{i-1}+1}, \gamma_{g_0+ \cdots +g_{i-1}+1}\right] \cdots \left[\alpha_{g_0+ \cdots + g_i}, \gamma_{g_0+ \cdots + g_i}\right]\right)\\
&=\sum n_i g_{j*}\left( \left[\alpha_{g_0+ \cdots+g_{i-1}+1}, \gamma_{g_0+ \cdots +g_{i-1}+1}\right] \cdots \left[\alpha_{g_0+ \cdots + g_i}, \gamma_{g_0+ \cdots + g_i}\right]\right)\\
&=n_j \left[g_{j*}(\alpha_{g_0+ \cdots+g_{j-1}+1}), g_{j*}(\gamma_{g_0+ \cdots +g_{j-1}+1})\right]\\
& \hspace{5.7cm}\cdots \left[g_{j*}(\alpha_{g_0+ \cdots + g_j}), g_{j*}(\gamma_{g_0+ \cdots + g_j})\right].
\end{align*}
Clearly the resulting term
$$
n_j \left[g_{j*}(\alpha_{g_0+ \cdots+g_{j-1}+1}), g_{j*}(\gamma_{g_0+ \cdots +g_{j-1}+1})\right] \cdots \left[g_{j*}(\alpha_{g_0+ \cdots + g_j}), g_{j*}(\gamma_{g_0+ \cdots + g_j})\right]
$$
is nonzero in $H_1(T)$.  Hence $ i_*(\sum n_i y_i) \ne 0$.

\end{proof}

 Let $f':S \rightarrow S$ be constructed by taking a map $f$ in $J_k(D)$ and extending it to the whole surface by the identity, as in \lem{inclusion}. We relate the $\tau_k(f)$ and $\tau'_k(f')$ in the following lemma.

\begin{lemma}\label{L:inclusionnonzero}
Let $S$ be a surface with genus $g \ge 2$ and 1 boundary component.  Let $D$ be a sphere with $n$ disks removed, $n \ge 3$.  Let $i:D \rightarrow S$ be a separating embedding.  Let $f \in J_k(D)$ with $\tau_k(f)(A_i)=w_i\in \pi_1(D)_k/\pi_1(D)_{k+1}$ and let $f'=i'(f)$ be the element of $Mod(S)$ given by \lem{inclusion}.  Let $\gamma_i$ and $\gamma_j$ be elements of $F$ that intersect $i(D)$ along the arcs $A_i$ and $A_j$ respectively.  Then  $\tau'_k(f')[\gamma_i, \gamma_j]=(1-\gamma_i)(1- \gamma_j)\left(i_*(w_i)-i_*(w_j)\right)$.  Furthermore, if $w_i \ne w_j$ as elements of $\pi_1(D)_k/\pi_1(D)_{k+1}$ for some choice of $i,j$ then $\tau'_k(f') \ne 0$.
\end{lemma}

\begin{proof}

Let $w_i= \tau_k(f)(A_i)$.  We wish to show that $\tau'_k(f')[\gamma_i, \gamma_j] \ne 0$.  We begin by computing $\tau'_k(f')([\gamma_i, \gamma_j])$.  By construction, $\gamma_i=CA_i\beta_i\overline{A_i}\,\overline{C}$ for some loop $\beta_i$ in $S \setminus D$ by our construction of the basis for $F$.  Then by definition
\begin{align*}
\tau'_k(f')([\gamma_i, \gamma_j]) & =f'([\gamma_i, \gamma_j])[\gamma_j, \gamma_i]\\
&=[f'(\gamma_i), f'(\gamma_j)][\gamma_j, \gamma_i]\\
&=[f'(CA_i\beta_i\overline{A_i}\,\overline{C}), f'(CA_j\beta_j\overline{A_j}\overline{C})][\gamma_j, \gamma_i]\\
&=[C\, if(A_i)\beta_i if(\overline{A_i})\overline{C},C\, if(A_j)\beta_j if(\overline{A_j})\overline{C} ] [\gamma_j, \gamma_i]\\
&=[C\,if(A_i)\beta_i\overline{if(A_i)}\overline{C}, C\,if(A_j)\beta_j \overline{if(A_j)}\overline{C} ] [\gamma_j, \gamma_i]\\
&=\left[C\,i\left(f(A_i)\overline A_i\right) A_i\beta_i\overline{A_i}i\left(A_i\overline{f(A_i)}\right)\overline{C},\right.\\
&\left. \hspace{.7in} C\,i\left(f(A_j)\overline{A_j}\right) A_j \beta_j \overline{A_j}i\left(A_j\overline{f(A_j)}\right)\overline{C} \right] [\gamma_j, \gamma_i]\\
&=\left[\left(C\,i\left(f(A_i)\overline A_i\right)\overline{C}\right) \left(C A_i\beta_i\overline{A_i}\,\overline{C}\right) \left(C \,i\left(A_i\overline{f(A_i)}\right)\overline{C}\right),\right.\\
&  \left. \left(C\,i\left(f(A_j)\overline{A_j}\right)\overline{C}\right) \left(C A_j \beta_j \overline{A_j}\overline{C}\right) \left(Ci\left(A_j\overline{f(A_j)}\right)\overline{C} \right)\right] [\gamma_j, \gamma_i].
\end{align*}

Note that $i_*:\pi_1(D, p_0) \rightarrow \pi_1(S,i(p_0))$.  Allowing a change of basepoint from $\pi_1(S,i(p_0))$ to $\pi_1(S,*)=F$, we may further reduce this expression as follows:
\begin{align*}
\tau'_k(f')([\gamma_i, \gamma_j]) &=[i_*(w_i)\gamma_i i_*(w_i^{-1}), i_*(w_j) \gamma_j i_*(w_j^{-1})][\gamma_j, \gamma_j]\\
&=[[i_*(w_i),\gamma_i] \gamma_i, [i_*(w_j), \gamma_j]\gamma_j][\gamma_j, \gamma_i]
\end{align*}

Using the two commutator identities $[ga,b]=\lup{g}[a,b][g,b]$ and \linebreak $[a,hb]=[a,h] \lup{h}[a,b]$ it is possible to reduce this expression to the following:

\begin{align*}
\tau'_k(f')([\gamma_i,\gamma_j])=&\lup{[i_*(w_i),\gamma_i]}[\gamma_i,[i_*(w_j),\gamma_j]]\lup{[i_*(w_i),\gamma_i][i_*(w_j),\gamma_j]}[\gamma_i, \gamma_j]\\
&\qquad [[i_*(w_i),\gamma_i], [i_*(w_j),\gamma_j]] \lup{[i_*(w_j),\gamma_j]}[[i_*(w_i),\gamma_i],\gamma_j][\gamma_j, \gamma_i].
\end{align*}

As $ i_*(G) \subset F'$ and $w_i,w_j\in G_k$, $i_*(w_i), i_*(w_j) \in F'_k$.  Thus the \linebreak commutators $[i_*(w_i), \gamma_i], [i_*(w_j), \gamma_j]$ are elements of $F'_k$ and hence the \linebreak conjugation in our expression is trivial modulo $F'_{k+1}$.  In addition \linebreak $[[i_*(w_i),\gamma_i], [i_*(w_j),\gamma_j]] \in F'_{k+1}$.  Thus reducing mod $F'_{k+1}$ we obtain:

\begin{align*}
\tau'_k(f')([\gamma_i,\gamma_j])&=[\gamma_i,[i_*(w_j),\gamma_j]][\gamma_i, \gamma_j][[i_*(w_i),\gamma_i],\gamma_j][\gamma_j, \gamma_i]\\
&=[\gamma_i,[i_*(w_j),\gamma_j]][[i_*(w_i),\gamma_i],\gamma_j]
\end{align*}

Equivalently, viewing the $\tau'_k(f')([\gamma_i, \gamma_j])$ as an element of the $\mathbb{Z}[F/F']$ module we can represent it as follows: 
$$
\tau'_k(f')[\gamma_i, \gamma_j]=(1-\gamma_i)(1- \gamma_j)\left(i_*(w_i)-i_*(w_j)\right)
$$
where $(1-\gamma_i),(1-\gamma_j) \in \Z[F/F]$ and $i_*(w_i),i_*(w_j) \in F'_k/F'_{k+1}$.  This proves the first statement of the lemma.

To prove that $w_iw_j^{-1} \ne 0$ shows $\tau'_k(f')\ne 0$, we find it advantageous to express the above computation as follows:
$$
\tau'_k(f')[\gamma_i, \gamma_j]=(1-\gamma_i)(1- \gamma_j)\left(i_*(w_i w_j^{-1})\right).
$$
By \lem{torsion}, $\tau'_k(f')[\gamma_i, \gamma_j]$ is nonzero provided that $i_*(w_i w_j^{-1})$ is nontrivial.  This follows directly from \lem{braidsinject}.
\end{proof}


Let $D_n$ be the disk with $n$ holes and let $G(n)=\pi_1(D_n)$.  Let $E(n)$ denote the free group generated by $x_1, x_2, \dots, x_n$.  Let $P(n)$ denote the pure braid group on $n$ strands.  Consider the inclusion $\iota:  E(n-1) \rightarrow P(n)$ obtained by mapping the generator $x_i$ of $E(n-1)=\langle x_1, \dots ,x_{n-1}\rangle $ by $x_i \mapsto A_{i,n}$ where $A_{i,n}$ is the generator of the pure braid group which clasps strands $i$ and $n$ \cite{B} as shown in \figr{braid generator}.  

Forgetting to fix the boundary components in the interior of the disk, any mapping class in $Mod(D)$ is isotopic to the identity.  The trace of this isotopy permutes the boundary components on the interior of the disk to generate a pure braid.  This correspondence is an isomorphism between $P(n)$ and $Mod(D_n)$.  We denote this natural map $\psi:P(n) \rightarrow Mod(D_n)$.  In particular it is important to note that the pure braid generator $A_{i,n}$ yields a mapping class $f_{i,n}$ on $D_n$ given by a single dehn twist (twisting right) around the $i^{th}$ and $n^{th}$ punctures as shown in \figr{braid generator}.   
\begin{figure}[h]
\begin{center}
\includegraphics[scale=.19]{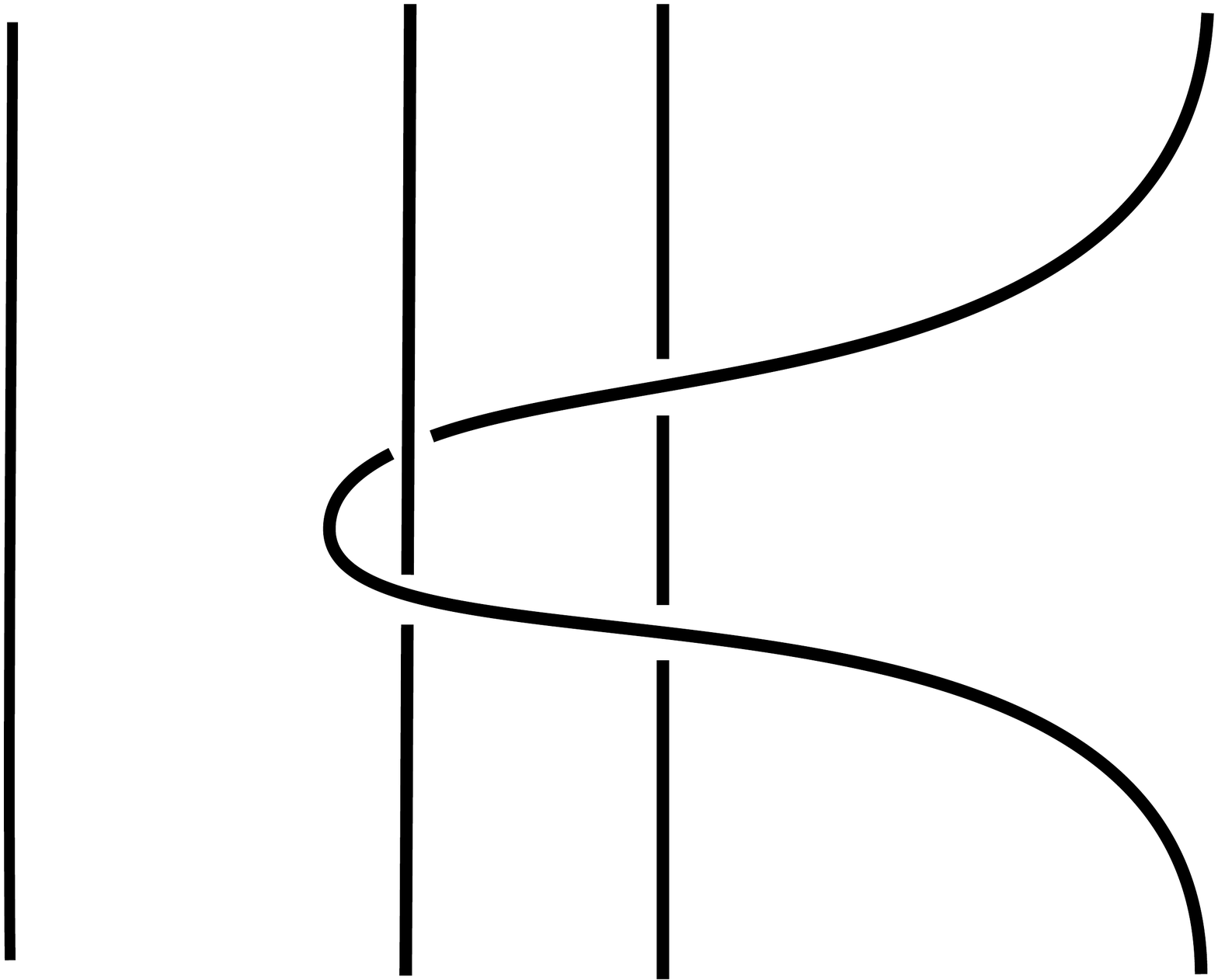} \quad \includegraphics[scale=.21]{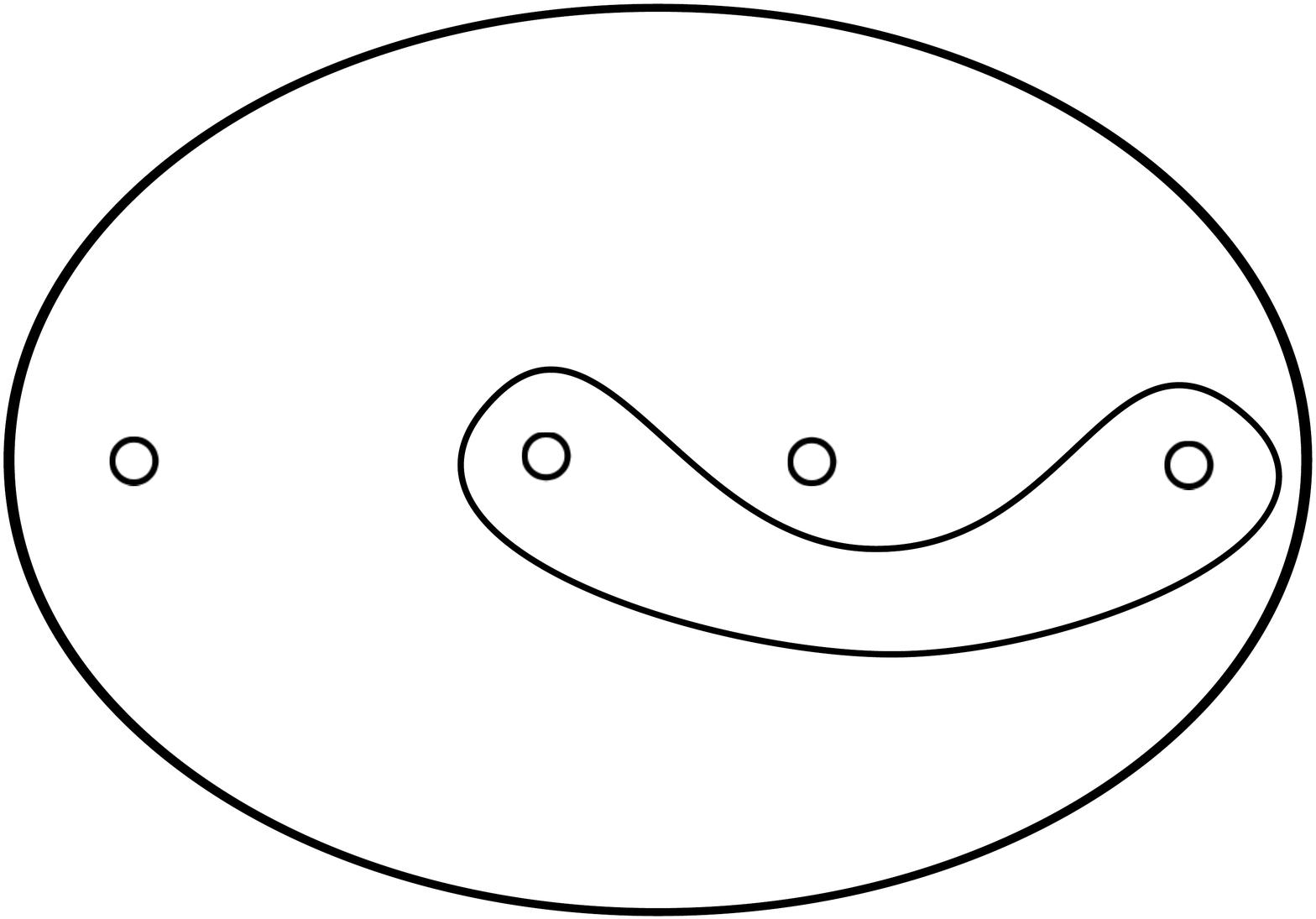}
\put(-347,115){{\small 1}}
\put(-329,115){{\small $\cdots$}}
\put(-300,115){{\small $i$}}
\put(-280,115){{\small $i+1$}}
\put(-240,115){{\small $\cdots$}}
\put(-209,115){{\small $n$}}
\put(-162,75){{\small 1}}
\put(-138,75){{\small $\cdots$}}
\put(-107,75){{\small $i$}}
\put(-80,75){{\small $i+1$}}
\put(-45,75){{\small $\cdots$}}
\put(-18,75){{\small $n$}}
\caption[The generator $A_{i,n}$ of the pure braid group and the Dehn twist $f_{i,n}$ corresponding to $A_{i,n}$]{Left: The generator $A_{i,n}$ of the pure braid group.  Right: The Dehn twist $f_{i,n}$ corresponding to $A_{i,n}$ }
\label{F:braid generator}
\end{center}
\end{figure}
Note that as function composition is written right to left, the map $\psi$ acts by reversing the order of pure braid generators: $\psi(A_{p_1,n}^{\epsilon_1} \cdots A_{p_m,n}^{\epsilon_m})=f_{p_m,n}^{\epsilon_m} \cdots f_{p_1,n}^{\epsilon_1}$. 

For a mapping class $f \in Mod(D_n)$, let $\phi_i(f)$ be given by $\phi_i(f)=f(A_i)\overline{A_i}$.

\begin{lemma}\label{L:identity}
The map $\theta:E(n-1) \rightarrow G(n-1)$ given by the composition of maps
$$
E(n-1) \stackrel{\iota}{\hookrightarrow} P(n) \stackrel{\psi}{\hookrightarrow} Mod(D_n) \stackrel{\phi_n}{\rightarrow} G(n)\stackrel{\pi}{ \rightarrow} G(n-1)
$$
 is the isomorphism induced by mapping $x_i \mapsto y_i$.

The map $\mu:E(n-1) \rightarrow G(n-1)$ given by the composition of maps
$$
E(n-1) \stackrel{\iota}{\hookrightarrow} P(n) \stackrel{\psi}{\hookrightarrow} Mod(D_n) \stackrel{\phi_1}{\rightarrow} G(n) \rightarrow G(n-1)
$$
 is the homomorphism given by $v \mapsto y_1^{\eta}$ where $\eta$ is the sum of the $x_1$  exponents in $v$.

\end{lemma}

\begin{proof}
To show this it suffices to trace $v \in E(n-1)$ through the above maps.  By the above definitions it is clear that for $v=x_{p_1}^{\epsilon_1} \cdots x_{p_m}^{\epsilon_m}$ that $\left(\psi \circ \iota\right) (v)=f_{p_m,n}^{\epsilon_m} \cdots f_{p_1,n}^{\epsilon_1}$.  Let $\left(\psi \circ \iota\right) (v)=f$.

To compute $\phi_n(f)$ and $\phi_1(f)$ we examine the image of the arcs $A_1$ and $A_n$, and the generators of $G(n)$ under a map $f_{i,n}$.  By direct computation we find that:

\begin{align*}
f_{i,n}(A_1)& \simeq {\begin{cases} A_1 \qquad & \text{ if } 1<i \\ y_n y_1 A_1 & \text{ if } i=1 \end{cases}} \\
f_{i,n}(A_n)& \simeq y_n y_i A_n \\
f_{i,n}(y_n)& \simeq y_n y_i y_n y_i^{-1}y_n^{-1}\simeq y_n[y_i,y_n] \\
f_{i,n}(y_i)& \simeq y_n y_i y_n^{-1} \simeq y_i[y_i^{-1},y_n]\\
f_{i,n}(y_j)&\simeq
{\begin{cases} [y_i,y_n]^{-1} y_j [y_i,y_n]  \qquad & \text{if } i < j, j \ne g\\ y_j \qquad \qquad \qquad \qquad \,\,\,& \text{if } i> j \end{cases}} 
\end{align*}

Similarly, we can compute the image of the arcs $A_1$ and $A_n$, and the generators of $G(n)$ under the map $f_{i,n}^{-1}$, the left handed Dehn twist about the same simple closed curve.

\begin{align*}
f_{i,n}^{-1}(A_1)& \simeq {\begin{cases} A_1 \qquad & \text{ if } 1<i \\ y_1^{-1} y_n^{-1} A_1 & \text{ if } i=1 \end{cases}} \\
f_{i,n}^{-1}(A_n)& \simeq y_i^{-1} y_n^{-1} A_n \\
f_{i,n}^{-1}(y_n)& \simeq y_i y_n^{-1} y_n y_n y_i\simeq [y_i^{-1},y_n] y_n \\
f_{i,n}^{-1}(y_i)& \simeq y_i^{-1} y_n^{-1} y_i y_n y_i \simeq [y_i^{-1},y_n^{-1}] y_i\\
f_{i,n}^{-1}(y_j)&\simeq
{\begin{cases} [y_i^{-1},y_n^{-1}] y_j [y_i^{-1},y_n^{-1}]^{-1}  \qquad & \text{if } i < j, j \ne g\\ y_j \qquad \qquad \qquad \qquad \,\,\,& \text{if } i> j \end{cases}} 
\end{align*}

Let $N$ be the normal subgroup of $G(n)$ normally generated by $y_n$.  We can rewrite the above computations as follows.  We use $f_{i,n}$ and $f_{i,n}^{-1}$ to denote both the mapping classes and their induced map on $G(n)$ for convenience of notation.

\begin{align*}
f_{i,n}(A_1)\overline{A_1}& \simeq{\begin{cases} 1 \qquad & \text{ if } 1<i \\ y_1 & \text{ if } i=1 \end{cases}} \\
f_{i,n}(A_n)\overline{A_n}&\simeq y_i \\
f_{i,n}(y_n)&\simeq 1\\
f_{i,n}(y_j)&\simeq y_j  \qquad \qquad \text{if }  j \ne g
\end{align*}

For the inverse map $f_{i,n}^{-1}$ we compute:

\begin{align*}
f_{i,n}^{-1}(A_1)\overline{A_1}& \simeq{\begin{cases} 1 \qquad & \text{ if } 1<i \\ y_1^{-1} & \text{ if } i=1 \end{cases}} \\
f_{i,n}^{-1}(A_n)\overline{A_n}&\simeq y_i^{-1} \\
f_{i,n}^{-1}(y_n)&\simeq 1\\
f_{i,n}^{-1}(y_j)&\simeq y_j  \qquad \qquad \text{if }  j \ne n
\end{align*}

From the above it is clear that $f_{i,n}$ and $f_{i,n}^{-1}$ act by the identity on the first $n-1$ generators of $G$, and sends $y_n$ to an element of the subgroup $N$.   Note that the map $G \mapsto G/N$ is the homomorphism of $\pi_1(D_n)$ induced by the map $\ell$ which caps the $n^{th}$ boundary component of $D_n$ as in \figr{capping}.
\begin{figure}[h] 
   \centering
   \includegraphics[width=2.24in]{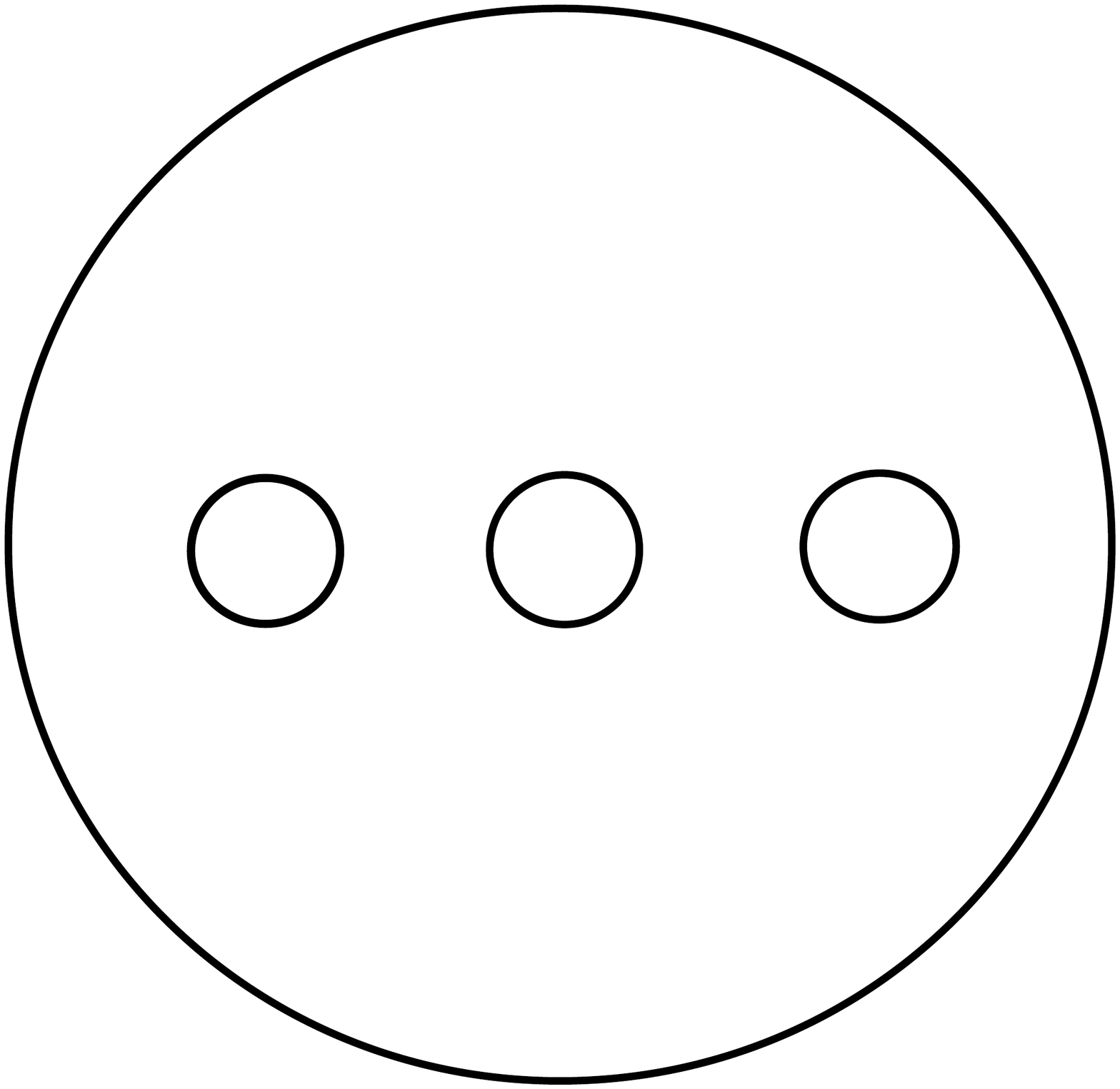}  \hspace{1cm}\includegraphics[width=2.24in]{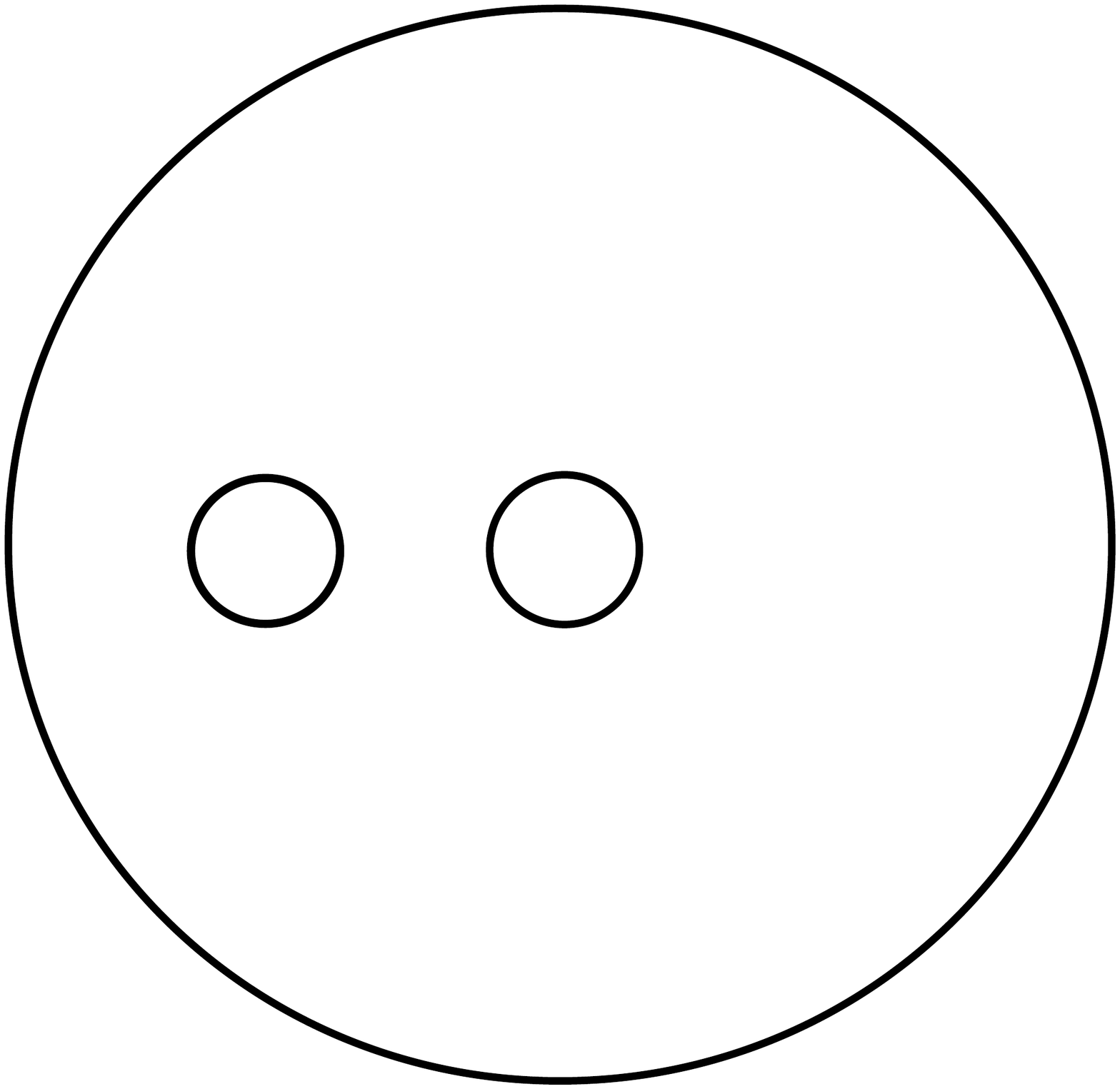} 
   \put(-200,50){$\xrightarrow{\hspace*{1cm}}$}
   \put(-297,55){\scriptsize $\dots$}
      \put(-103,55){\scriptsize $\dots$}
   \caption[The map $\ell:D_n \rightarrow D_{n-1}$]{Above is an illustration of the map $\ell:D_n \rightarrow D_{n-1}$ obtained by capping off the $n^{th}$ boundary component.  From this one can see $\pi_1(D_{n-1})=\langle y_1, \dots y_{n-1}\rangle \cong G/N$.}
   \label{F:capping}
\end{figure}
As the mapping classes $f_{i,n}$ and $f_{i,n}^{-1}$ fix the $n^{th}$ boundary component and $\ell$ is an inclusion map, the map $\ell$ commutes with $f_{i,n}$ and $f_{i,n}^{-1}$.  Hence \linebreak $N=f_{i,n}(N)=f_{i,n}^{-1}(N)$.  Thus, given a word in $v\in G$, $f_{i,n}$ and $f_{i,n}^{-1}$ each map $v$ to a word of the same class in $G/N$.  

We show by induction that for $f=f_{p_m,n}^{\epsilon_m} \cdots f_{p_1,n}^{\epsilon_1}$ the following computations hold  mod $N$:
\begin{align*}
f(A_1)\overline{A_1}&=y_1^\eta\\
f(A_n)\overline{A_n}&=y_{p_1}^{\epsilon_1} \cdots y_{p_m}^{\epsilon_m}\\
\end{align*}
where $\eta=\sum_{p_i=1} \epsilon_i$.  

For simplicity of notation we rewrite these equalities as
\begin{align*}
f(A_1)\overline{A_1}&=y_1^{\pm 1} \cdots y_1^{\pm 1}\\
f(A_n)\overline{A_n}&=y_{l_1}^{\pm1} \cdots y_{l_k}^{\pm1}\\
\end{align*}
where indexes are allowed to repeat.  The initial case of the induction was done by previous computations.  Suppose that the above computations hold.  Then it follows that
\begin{align*}
f(A_1)&\simeq a_{0}y_1^{\pm 1}\cdots y_1^{\pm 1} A_1\\
f(A_n)&\simeq a_{0}' y_{l_1}^{\pm1} \cdots y_{l_k}^{\pm1}A_n.\\
\end{align*} 
where $a_0, a_0' \in N$.  Consider $f_{p_{m+1},n}^{\pm 1}f_{p_m,n}^{\epsilon_m} \cdots f_{p_1,n}^{\epsilon_1}=f_{p_{m+1},n}^{\pm 1}f$.  As $f_{p_{m+1},n}^{\pm 1}$ acts by the identity on $G/N$, we have that 

\begin{align*}
f_{p_{m+1},n}(y_{i})=a_{i}^+y_{i}\\
f_{p_{m+1},n}^{-1}(y_{i})=a_{i}^-y_{i}\\
\end{align*}
for some $a_{i}^+, a_{i}^- \in N$, for all $i$.   Hence we compute
\begin{align*}
f_{p_{m+1},n}f(A_1)&\simeq \begin{cases} f_{p_{m+1},n}(a_{0})\left(a_1^+ y_1\right)^{\pm 1}\cdots \left( a_1^+y_1\right)^{\pm 1} A_1 & \text{if }p_{m+1}\ne 1\\ f_{p_{m+1},n}(a_{0})\left(a_1^+ y_1\right)^{\pm 1}\cdots \left( a_1^+y_1\right)^{\pm 1} y_n y_1A_1 & \text{if }p_{m+1}=1 \end{cases}\\
f_{p_{m+1},n}^{-1}f(A_1)&\simeq \begin{cases} f_{p_{m+1},n}^{-1}(a_{0})\left(a_1^- y_1\right)^{\pm 1}\cdots \left( a_1^-y_1\right)^{\pm 1} A_1 & \text{if }p_{m+1}\ne 1\\ f_{p_{m+1},n}^{-1}(a_{0})\left(a_1^- y_1\right)^{\pm 1}\cdots \left( a_1^-y_1\right)^{\pm 1} y_1^{-1}y_n^{-1}A_1 & \text{if }p_{m+1}=1 \end{cases}\\
f_{p_{m+1},n}f(A_n)&\simeq f_{p_{m+1},n}(a_{0}') \left(a_{l_1}^+y_{l_1}\right)^{\pm1} \cdots \left(a_{l_k}^+y_{l_k}\right)^{\pm1}y_n y_{p_{m+1}}A_n\\
f_{p_{m+1},n}^{-1}f(A_n)&\simeq f_{p_{m+1},n}(a_{0}') \left(a_{l_1}^-y_{l_1}\right)^{\pm1} \cdots \left(a_{l_k}^-y_{l_k}\right)^{\pm1} y_{p_{m+1}}^{-1}y_n^{-1}A_n.\\
\end{align*} 

Therefore, as $f_{p_{m+1},n}^{\pm 1}(a_{0}), f_{p_{m+1},n}^{\pm 1}(a_{0}'), a_i^+, a_i^-, y_n, y_n^{-1} \in N$, we can do the following computation mod $N$.
\begin{align*}
f_{p_{m+1},n}^{\pm 1}f_{p_m,n}^{\epsilon_m} \cdots f_{p_1,n}^{\epsilon_1}(A_1)\overline{A_1}&= \begin{cases} y_1^\eta y_1^{\pm 1} & \text{if } {p_{m+1}=1}\\ y_1^\eta & \text{if } {p_{m+1}\ne1} \end{cases}\\
f_{p_{m+1},n}^{\pm 1}f_{p_m,n}^{\epsilon_m} \cdots f_{p_1,n}^{\epsilon_1}(A_n)\overline{A_n}&=y_{p_1}^{\epsilon_1} \cdots y_{p_m}^{\epsilon_m}y_{p_{m+1}}^{\pm 1}
\end{align*}
This completes the induction.

Thus $\theta(x_{p_1}^{\epsilon_1} \cdots x_{p_m}^{\epsilon_m})=y_{p_1}^{\epsilon_1} \cdots y_{p_m}^{\epsilon_m}$ and $\mu(w)=y_1^{\eta}$, as desired.

\end{proof}

Note that for words $v \in [E(n-1), E(n-1)]$ the maps $f_{1,n}$ occur in pairs with opposite exponents.  Hence for $v \in [E(n-1), E(n-1)]$, $\mu(v)=1$.


\subsection{Structure of the Magnus subgroup quotients}\label{S:mainthm}

In Lemma \ref{L:identity} we considered compositions of maps which defined a correspondence between elements of the free group $E(n-1)$ and elements of $Mod(D_n)$.  Lemma \ref{L:inclusionnonzero} allows us to relate the Johnson homomorphism $J_k(D_n)$ to the Magnus homomorphism $M_k(S_g)$.  We now combine these tools to construct families of mapping classes in $M_k(S_g)$ which have a desirable algebraic structure in the image of the Magnus homomorphism.

Let $i:D_g \rightarrow S_g$ be the separable embedding illustrated in \figr{embedding}.  Consider the following composition of maps:
$$
E(g-1) \stackrel{\iota}{\hookrightarrow} P(g) \stackrel{\psi}{\hookrightarrow} Mod(D_g) \stackrel{i'}{\rightarrow} Mod(S_g)
$$
where $i':Mod(D_g) \rightarrow Mod(S_g)$ is the map described in Lemma \ref{L:inclusion}.  Let $\rho=i' \circ \psi \circ \iota$.  The following theorem illustrates that $\rho$ retains the structure of the free group.

\begin{theorem}\label{T:containsbraids}
Let $S$ be an orientable surface with genus $g \ge 3$.  Then the map $\rho: E(g-1) \rightarrow Mod(S)$ induces a monomorphism on the quotients $\overline{\rho}:E(g-1)_k/ E(g-1)_{k+1} \rightarrow M_k(S)/M_{k_1}(S)$ for all $k$.
\end{theorem}

\begin{proof} Let $D_g$ be a disk with $g$ punctures.  To prove the theorem it suffices to show that mapping classes contained in the subgroup $\rho(E(g-1))$ satisfy the conditions of \lem{inclusionnonzero} and produce distinct images through the Magnus homomorphism.  For this we employ several results about the pure braid group, $P(g)$.   Consider the following split exact sequence 

$$1 \stackrel{}{\rightarrow} E(g-1) {\rightarrow} P(g){ \rightarrow} P(g-1) \rightarrow 1$$

where the map $E(g-1) \rightarrow P(g)$ is as in \lem{identity} and $P(g) \rightarrow P(g-1)$ is given by forgetting the $g^{th}$ strand. This exact sequence induces an isomorphism as given in \cite{FR}:

$$\frac{E(g-1)_k}{E(g-1)_{k+1}} \oplus \frac{P(g-1)_k}{P(g-1)_{k+1}} \cong \frac{P(g)_k}{P(g)_{k+1}}$$

In particular, the map $\iota$ induces an injective map $\overline{\iota}$ on the lower central series quotients:
$$\overline{\iota}:\frac{E(g-1)_k}{E(g-1)_{k+1}} \hookrightarrow \frac{P(g)_k}{P(g)_{k+1}}$$.

By direct analysis of the induced automorphisms on $G(g)$ \cite{B} \linebreak Corollary 1.8.3, it is clear that $\psi\left(P(g) \right)\subset J_2(D_g)$.  Given this, Lemma \ref{L:commutator} shows that $\psi\left(P(g)\right)_k \subset J_k(D_g)$.  Hence, we have a well \linebreak defined map $\overline{\psi}:\frac{P(g)_k}{P(g){k+1}} \rightarrow \frac{J_k(D_g)}{J_{k+1}(D_g)}$.  By \cite{GH}, Theorem 1.1 \linebreak $\psi \left(P(g)_{k+1}\right)=\psi \left(P(g) \right) \cap J_{k+1}(D_g)$.  Hence the map $\overline{\psi}$ is injective.

$$\overline{\psi}:\frac{P(g)_k}{P(g){k+1}} \hookrightarrow \frac{J_k(D_g)}{J_{k+1}(D_g)}$$

By \prop{homomorphism}, the map $i'$ induces a monomorphism \linebreak $\overline{i}:\frac{J_k(D_g)}{J_{k+1}(D_g)} \rightarrow \frac{J_k(D_g)}{J_{k+1}(D_g)}$.  By \lem{inclusionnonzero}, given $v \in \frac{E(g-1)_k}{E(g-1)_k}$ we have that 
$$
 \tau'_k(i'\psi\iota(v))[\gamma_1, \gamma_g]=(1-\gamma_1)(1- \gamma_g)i_*(w_1)i_*(w_g)^{-1}
$$ 
written as an element of $\frac{F'_k}{F'_{k+1}}$ as a $\Z \left[ \frac{F}{F'} \right]$ module where $w_i=\tau_k(f)(A_i)$.  


Note that we have traced $w \in E(g-1)_k/E(g-1)_{k+1}$ through the following composition of maps:
$$
\frac{E(g-1)_k}{E(g-1)_k} \stackrel{\overline{\iota}}{\hookrightarrow} \frac{P(g)_k}{P(g)_{k+1}}\stackrel{\overline{\psi}}{\hookrightarrow} \frac{J_k(D_g)}{J_{k+1}(D_g)} \stackrel{\overline{i}}{\hookrightarrow}\frac{M_k(S_g)}{M_{k+1}(S_g)} \stackrel{\tau'_k(-)[\gamma_1,\gamma_g]}{\longrightarrow}\frac{F'_k}{F'_{k+1}}
$$.

By definition, the maps $\overline{\iota}$ and $\overline{\psi}$ are homomorphisms.  \prop{homomorphism} shows that $\overline{i}$ is a homomorphism.  The map $\tau'_k(-)[\gamma_1,\gamma_g]: M_k(S) \rightarrow F'_k/F'_{k+1}$ is a homomorphism as $\tau'_k$ is a homomorphism.  As all maps in this composition are homomorphisms,  the composition map is also a homomorphism.

As this composition is a homomorphism, in order to complete the proof it suffices to show that the image of $v$ through this composition is not \linebreak the identity.  As shown in Lemma \ref{L:torsion}, this module has no torsion of \linebreak the form  $(1-\gamma)x=0$ where $\gamma$ is a generator of $F$. Hence for \linebreak $i_*(w_1)i_*(w_g)^{-1}=i_*(w_1w_g^{-1})\ne 1$ as an element of $F'_k/F'_{k+1}$, we have that $ \tau'_k(i'\psi\iota(w))[\gamma_1, \gamma_g]=(1-\gamma_1)(1- \gamma_g)i_*(w_1w_g^{-1}) \ne 0$.  By \lem{braidsinject} the map $i_*$ is injective, thus it suffices to show that $w_1w_g^{-1}\ne1$ as an element of $ \frac{G(g)_k}{G(g)_{k+1}}$.  By \lem{identity},
\begin{align*}
\pi(w_1 w_g^{-1})&=\pi(w_1) \pi(w_g)^{-1}&\\
&= \mu(v) \theta(v)^{-1}&\\
&=v^{-1} \qquad \qquad & \text{when written in the generators $y_i$ of $G(g-1)$.}\\
\end{align*}
Since ${\pi}$ is a homomorphism we can conclude that  $w_1w_g^{-1} \ne 1$ and hence  $ \tau'_k(i'\psi\iota(w))[\gamma_1, \gamma_g]=(1-\gamma_1)(1- \gamma_g)i_*(w_1w_g^{-1}) \ne 0$.  This shows that \linebreak $\ker(\overline{\rho})=0$ and hence $\overline{\rho}$ is injective. 

\end{proof}

\begin{theorem}\label{T:infinitelygenerated}
Let $S$ be an orientable surface with genus $g \ge 3$.  Then the successive quotients of the Magnus filtration $\frac{M_k(S)}{M_{k+1}(S)}$ surject onto an infinite rank torsion free abelian subgroup of $\frac{F'_k}{F'_{k+1}}$ via the map  
$$
\frac{M_k(S)}{M_{k+1}(S)} \stackrel{\tau'_k(-)[c_6,c_2]}{\longrightarrow}\frac{F'_k}{F'_{k+1}}
$$
where $c_2$ and $c_6$ are the generators of $F$ illustrated in \figr{infgenbasis}.
\end{theorem}

\begin{proof}
Let $\gamma$ and $\delta_n$ be the simple closed curves on $S$ shown in \figr{cfrep}.  Let $i_n:D \rightarrow S$ be the embedding which sends the 3 holed disk $D$ to a neighborhood $\gamma \cup \delta_n$.  This set of embeddings of the disk onto $S$ was used by Church and Farb in \cite{CF}, Theorem 3.2 to produce an infinite family of mapping classes in $Mag(S)$.  We employ the same embeddings to produce an infinite family of mapping classes in $M_k(S)$.

\begin{figure}[h] 
   \centering
   \includegraphics[width=4.5in]{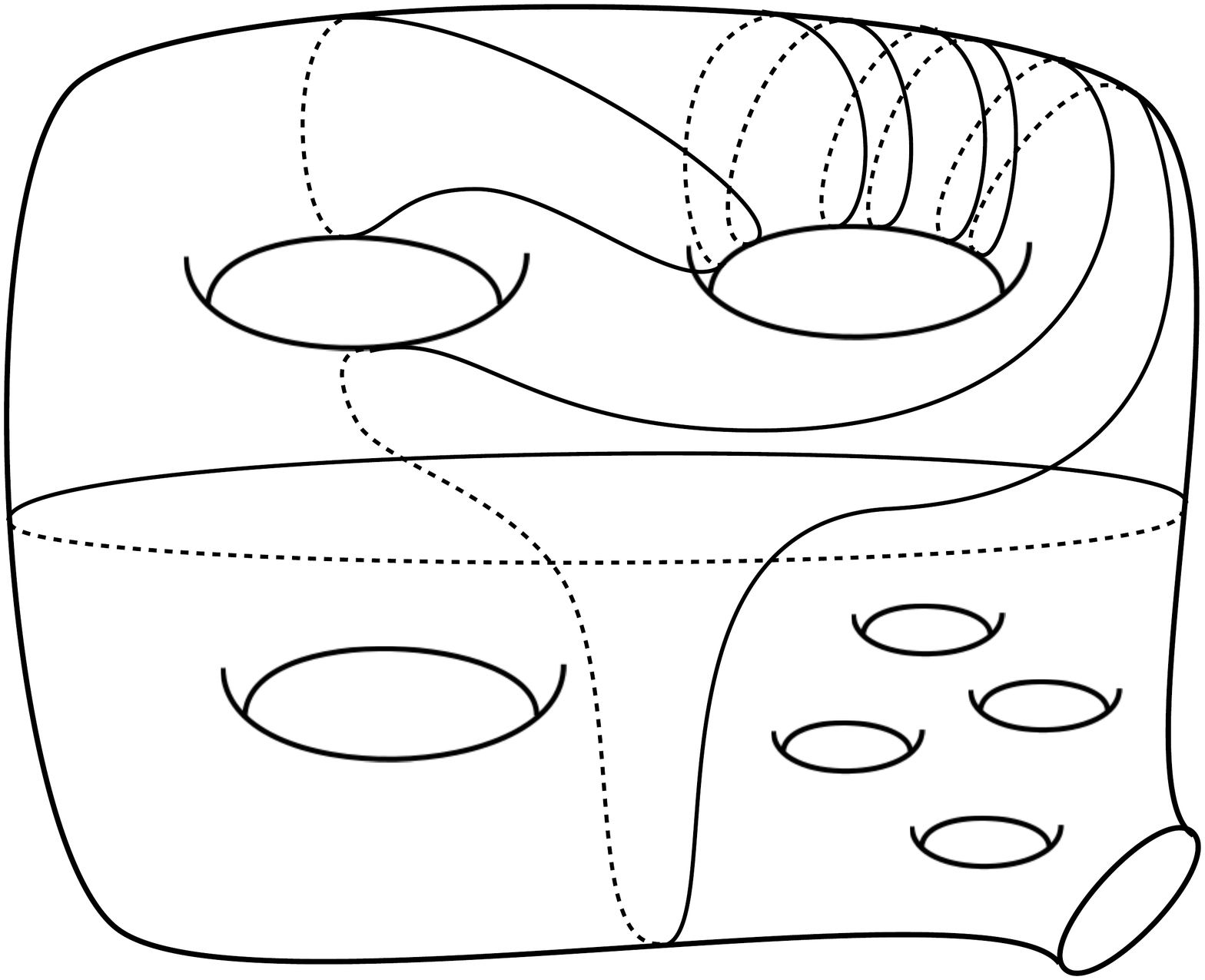}
   \put(-248,195){$\delta_3$}
   \put(-280,122){$\gamma$}
   \caption[The embeddings $i_n:D \rightarrow S$.]{Pictured above are two simple closed curves $\gamma$ and $\delta_3$.  The curve $\delta_n$ wraps $n$ times around the upper right handle.  We consider disks with 3 holes embedded by maps $i_n$  which send $D$ to a neighborhood of $\gamma \cup \delta_n$. }
   \label{F:cfrep}
\end{figure}

Let the free group $E(2)$ be generated by $\{x_1,x_2\}$.  Consider the commutator $c^k=[ \cdots [[x_2,x_1],x_1], \cdots ,x_1]\in E(2)_k$  (commutator with $x_1$ $k-1$ times).  By \cite{GH} Theorem 1.1, this commutator yields a nontrivial element of $\frac{J_k(D)}{J_{k+1}(D)}$ through the composition:
$$
E(2) \stackrel{\iota}{\hookrightarrow} P(3) \stackrel{\psi}{\hookrightarrow} Mod(D_3)
$$ 
Let $f_k$ be the mapping class in $\frac{J_k(D)}{J_{k+1}(D)}$ which arises from the commutator \linebreak $c^k$: $f_k=\iota \psi(c^k)$.  Let $i_n'(f_k)$ be the mapping class of $S$ resulting from extending $f_k$ by the identity on $S$ using the embedding $i_n$.

Each embedding $i_n:D \rightarrow S$ yields an infinite family of elements \linebreak $\tau'_k(i_n'(f_k))[c_6, c_2]$.  We will show that for each $k$ the set \linebreak $\{ \tau'_k(i_n'(f_k))[c_6, c_2]|n \in \N\}$ is independent in $\frac{F'_k}{F'_{k+1}}$ using the basis theorems developed in \sec{basistheorems}.

We begin by choosing a basis for $F$ for our computations.  Our chosen basis is illustrated in \figr{infgenbasis}.  Note that the generators $c_2$ and $c_6$ intersect each embedding $i_n(D)$ along the arcs $A_1$ and $A_3$ respectively.  Hence, we may compute $\tau'_k(i_n'(f_k))[c_6, c_2]$ as in \lem{inclusionnonzero}.

\begin{figure}[h] 
\centerline{\includegraphics[width=5.35in]{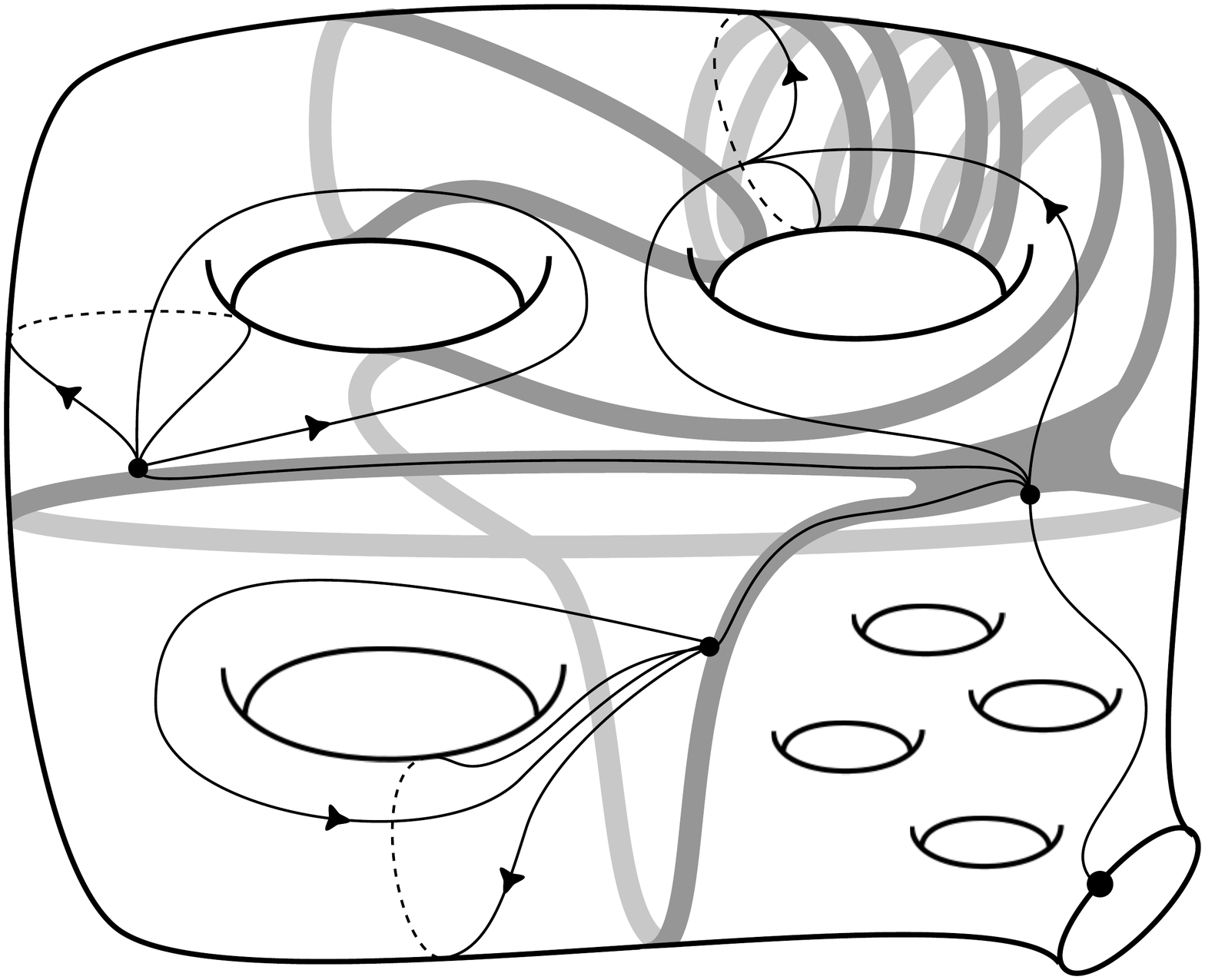}
   \put(-63,104){$C$}
   \put(-75,131){$p_0$}
   \put(-58,30){$*$}
   \put(-340,145){$p_3$}
   \put(-173,101){$p_1$}
	\put(-265,25){$c_2$}
	\put(-327,60){$c_1$}
		\put(-353,189){$c_6$}
		\put(-295,155){$c_5$}
	\put(-177,261){$c_4$}
	\put(-64,170){$c_3$}
		\put(-170,120){$A_1$}
		\put(-220,150){$A_3$}}
   \caption[The basis $\{c_1,c_2,c_3, c_4, c_5,c_6, \dots,c_{2g}\}$ chosen for computation of the Magnus homomorphisms. ]{The subsurface $i_3(D) \subset S$ is shown in grey.  The figure illustrates the basis $\{c_1,c_2,c_3, c_4, c_5,c_6, \dots,c_{2g}\}$ chosen for computation of the Magnus homomorphisms.  Note that $c_2$ and $c_6$ intersect $i_n(D)$ along the arcs $A_i$ as in \protect\lem{inclusionnonzero}.}
   \label{F:infgenbasis}
\end{figure}

By \lem{inclusionnonzero}, 
\begin{align*}
\tau'_k(i_n'(f_k))[c_6, c_2]=(1-c_6)(1-c_2)i_{n*}(w_3^k  (w_1^k)^{-1})
\end{align*}
where $w_i^k=f_k(A_i) \overline{A_i}$.  To show the set $\{\tau'_k(i_n'(f_k))[c_6, c_2]|n \in \N\}$ is independent we must compute the elements $i_{n*}(w_1)$ and $i_{n*}(w_3)$.  By \lem{torsion} the set  $\{\tau'_k(i_n'(f_k))[c_6, c_2]|n \in \N\}$ is independent if $\{i_{n*}(w_3^k (w_1^k)^{-1})|n \in \N\}$ is an independent set in $\frac{F'_k}{F'_{k+1}}$.  

We impose the following ordering the elements of our basis for $F'$: \linebreak $c_1<c_2<c_3<c_4<c_5<c_6$.  Then by \thm{Basis} the set \linebreak $B=\{\lup{w_{i,j}}[c_i, c_j] | w_{i,j} \in H_1 \left(E(c_1, \dots, c_j)\right)\}$ is a basis for $F'$.  By \cor{basistheorem}, for each $n \in N$, $i_{n*}(w_3^k (w_1^k)^{-1})$ can be expressed as a product of basic commutators of weight $k$ in the generators of $B$.  To show that the set $\{i_{n*}(w_3^k (w_1^k)^{-1})\}_{n \in \N}$ is independent we work towards expressing the elements as basic commutators in our basis $B$.

We denote the generators of $G$ which loop around the 3 interior boundary components of $D$  counterclockwise by $y_1,y_2,y_3$ as in \lem{identity}.  
As shown in \cite{CF}, Theorem 3.2, for the embedding $i_n:D \rightarrow S$ the generators of $G$ map to the following elements of $F$ written in terms of the basis chosen basis for $F$:

\begin{eqnarray*}
i_{n*}(y_1)&=& [c_2, c_1]\\
i_{n*}(y_2)&=& \lup{[c_5,c_6][c_3, c_4] c_4} [c_3 c_4^{-1} c_3^{-1} c_6, c_5 c_6^{n}]\\
i_{n*}(y_3)&=&[c_4, c_5 c_6^n c_3].
\end{eqnarray*}
Again, we have allowed a change of basepoint from $\pi_1(S,p_0)$ to $\pi_1(S, *)$ in this computation.
 
Recall that $\pi:G(3) \rightarrow G(2)$ is the map obtained by taking the quotient by the normal subgroup generated by $y_3$.  The retract $\pi:G(3) \rightarrow G(2)$ induces a retract of the lower central series quotients $\overline{\pi}:  \frac{G(3)_k}{G(3)_{k+1}} \rightarrow \frac{G(2)_k}{G(2)_{k+1}}$.  Let \linebreak $j: G(2) \rightarrow G(3)$ be the natural inclusion map.  Thus $\overline{\pi} \overline{j}:  \frac{G(2)_k}{G(2)_{k+1}}\rightarrow \frac{G(2)_k}{G(2)_{k+1}}$ is the identity map.
By \lem{identity} we have $\pi(w_1^k)=1$ and \linebreak $\pi(w_3^k)=[ \cdots [[y_2,y_1], y_1] \cdots,y_1]$.  Thus $\pi(w_3(w_1)^{-1})= \pi(w_3)$.  It then follows that $w_3^k(w_1^k)^{-1}=j \pi(w_3^k) \eta^k$ for some $\eta^k \in \ker \pi$.

We now compute the elements $i_{n*}j \pi(w_3^k (w_1^k)^{-1})$ using the above expressions for $i_{n*}(y_1)$ and $i_{n*}(y_2)$.
\begin{align*}
&i_{n*}j\pi(w_3^k(w_1^k)^{-1})=i_{n*}j\pi(w_3^k)\\
&=\left( i_{n*}\left([ \cdots [[y_2,y_1],y_1], \cdots,y_1]\right)\right)\\
&=\left([ \cdots [[i_{n*}(y_2),i_{n*}(y_1)], i_{n*}(y_1)], \cdots,i_{n*}(y_1)]\right)\\
&=\left(\Big{[} \cdots \Big{[}\Big{[}\lup{[c_5,c_6][c_3, c_4] c_4} [c_3 c_4^{-1} c_3^{-1} c_6, c_5 c_6^{n}],[c_2, c_1]\Big{]}, [c_2, c_1]\Big{]}, \cdots,[c_2, c_1]\Big{]}\right)&
\end{align*}


We will compute the elements $i_{n*}j \pi(w_3^k(w_1^k)^{-1})$ explicitly in terms of this basis $B$.  To do this we must reduce the expressions for $i_{n*}(y_2)$ to \linebreak products of basis elements of $F'$. Employing the commutator identity \linebreak $[ga,b]=\lup{g}[a,b][g,b]$ we can re-write the element $[c_3 c_4^{-1} c_3^{-1} c_6, c_5 c_6^{n}]$ as follows:
\begin{align*}
i_{n*}(y_2)&=[c_3 c_4^{-1} c_3^{-1} c_6, c_5 c_6^{n}]\\
&= \lup{c_3}[c_4^{-1} c_3^{-1} c_6, c_5 c_6^{n}][c_3, c_5 c_6^{n}]\\
&= \lup{c_3c_4^{-1}}[c_3^{-1} c_6, c_5 c_6^{n}]\lup{c_3}[c_4^{-1}, c_5 c_6^{n}][c_3, c_5 c_6^{n}]\\
&= \lup{c_3c_4^{-1}c_3^{-1}}[c_6, c_5 c_6^{n}]\lup{c_3c_4^{-1}}[c_3^{-1}, c_5 c_6^{n}]\lup{c_3}[c_4^{-1}, c_5 c_6^{n}][c_3, c_5 c_6^{n}]\\
\end{align*}

Using the commutator identity $[a,vb]=[a,v] \lup{v}[a,b]$, for any element $c$ we have:
\begin{align*}
[c, c_5 c_6^{n}]&=[c, c_5] \lup{c_5}[c,c_6^{n}]\\
&=[c, c_5] \lup{c_5}[c,c_6]\lup{c_5 c_6}[c,c_6^{n-1}]\\
&=[c, c_5] \lup{c_5}[c,c_6] \lup{c_5 c_6}[c,c_6]\cdots \lup{c_5 c_6^{n-1}}[c,c_6]
\end{align*}

Using this, our original expression becomes:
\begin{align*}
i_{n*}(y_2)=&\lup{c_3c_4^{-1}c_3^{-1}}[c_6, c_5] \lup{c_3c_4^{-1}c_3^{-1}c_5}[c_6,c_6] \lup{c_3c_4^{-1}c_3^{-1}c_5 c_6}[c_6,c_6]\cdots \lup{c_3c_4^{-1}c_3^{-1}c_5 c_6^{n-1}}[c_6,c_6]\\
&\lup{c_3c_4^{-1}}[c_3^{-1}, c_5] \lup{c_3c_4^{-1}c_5}[c_3^{-1},c_6] \lup{c_3c_4^{-1}c_5 c_6}[c_3^{-1},c_6]\cdots \lup{c_3c_4^{-1}c_5 c_6^{n-1}}[c_3^{-1},c_6]\\
&\lup{c_3}[c_4^{-1}, c_5] \lup{c_3c_5}[c_4^{-1},c_6] \lup{c_3c_5 c_6}[c_4^{-1},c_6]\cdots \lup{c_3c_5 c_6^{n-1}}[c_4^{-1},c_6]\\
&[c_3, c_5] \lup{c_5}[c_3,c_6] \lup{c_5 c_6}[c_3,c_6]\cdots \lup{c_5 c_6^{n-1}}[c_3,c_6].
\end{align*}
As $[c_6,c_6]=1$ this expression automatically reduces.  Using the identity $[a^{-1},b]=\lup{a^{-1}}[b,a]$ we simplify further to the following expression
\begin{align*}
i_{n*}(y_2)=&\lup{c_3c_4^{-1}c_3^{-1}}[c_6, c_5]\\
& \lup{c_3c_4^{-1}c_3^{-1}}[c_5, c_3] \lup{c_3c_4^{-1}c_5c_3^{-1}}[c_6,c_3] \lup{c_3c_4^{-1}c_5 c_6c_3^{-1}}[c_6,c_3]\cdots \lup{c_3c_4^{-1}c_5 c_6^{n-1}c_3^{-1}}[c_6,c_3]\\
&\lup{c_3c_4^{-1}}[c_5, c_4] \lup{c_3c_5c_4^{-1}}[c_6,c_4] \lup{c_3c_5 c_6c_4^{-1}}[c_6,c_4]\cdots \lup{c_3c_5 c_6^{n-1}c_4^{-1}}[c_6,c_4]\\
&[c_3, c_5] \lup{c_5}[c_3,c_6] \lup{c_5 c_6}[c_3,c_6]\cdots \lup{c_5 c_6^{n-1}}[c_3,c_6].
\end{align*}

Noting that $[a,b]=[b,a]^{-1}$ we can now write $i_{n*}(y_2)$ (additively) as follows:

\begin{align*}
i_{n*}(y_2)=&-\lup{c_3c_4^{-1}c_3^{-1}}[c_5, c_6]-\lup{c_3c_4^{-1}c_3^{-1}}[c_3, c_5] - \sum_{i=0}^{n-1}\lup{c_3c_4^{-1}c_5c_6^ic_3^{-1}}[c_3,c_6]\\
&-\lup{c_3c_4^{-1}}[c_4, c_5]-\sum_{i=0}^{n-1}\lup{c_3c_5c_6^i c_4^{-1}}[c_4,c_6] +[c_3, c_5]+ \sum_{i=0}^{n-1}\lup{c_5c_6^i}[c_3,c_6].
\end{align*}


By \prop{commutatoridentity}, for a fixed $n$ we may write $i_{n*}(j \pi w_3^k)$ in terms of basic commutators in the generators of $B$ as follows.
\begin{align*}
i_{n*}\left(j \pi w_3^k \right)=&-\left[ \cdots \left[\lup{c_3c_4^{-1}c_3^{-1}}[c_5, c_6], [c_2,c_1]\right], \cdots [c_2,c_1] \right]\\
&-\left[ \cdots \left[\lup{c_3c_4^{-1}c_3^{-1}}[c_3, c_5], [c_2,c_1]\right], \cdots [c_2,c_1] \right]\\
&- \sum_{i=0}^{n-1}\left[ \cdots \left[\lup{c_3c_4^{-1}c_5c_6^ic_3^{-1}}[c_3,c_6], [c_2,c_1]\right], \cdots [c_2,c_1] \right]\\
&-\left[ \cdots \left[\lup{c_3c_4^{-1}}[c_4, c_5], [c_2,c_1]\right], \cdots [c_2,c_1] \right]\\
&-\sum_{i=0}^{n-1}\left[ \cdots \left[\lup{c_3c_5c_6^i c_4^{-1}}[c_4,c_6] , [c_2,c_1]\right], \cdots [c_2,c_1] \right]\\
&+\left[ \cdots \left[[c_3, c_5], [c_2,c_1]\right], \cdots [c_2,c_1] \right]\\
&+ \sum_{i=0}^{n-1}\left[ \cdots \left[\lup{c_5c_6^i}[c_3,c_6], [c_2,c_1]\right], \cdots [c_2,c_1] \right].
\end{align*}

By \prop{splitting}, $\ker \overline{\pi}$ is generated by weight $k$ basic commutators in the generators $y_1,y_2, y_3$ with $y_3$ in at least one entry.  For convenience of notation, let us denote the elements of $B$ by $a_i$.  Let $A \subset B$ be the set of all elements $a_i$ such that $a_i$ appears with a nonzero coefficient in the expression for $i_{n*}(y_3)$ for some $n \in N$ when written in terms of the basis $B$.  Let $Y$ be the subgroup of $\frac{F'_k}{F'_{k+1}}$ generated by basic commutators with an entry from the set $A$.  Note that by construction, $i_{n*}(\ker\overline{\pi}) \subset Y$ for each $n$.  Hence if the elements $i_{n*}(w_3^k (w_1^k)^{-1})$ are independent in $\frac{F'_k/F'_{k+1}}{Y}$ they are also independent in ${F'_k/F'_{k+1}}$.
  Also notice that by construction, the group $\frac{F'_k/F'_{k+1}}{Y}$ is a free abelian group generated by basic commutators in elements of $B \setminus A$.  To consider whether the elements $i_{n*}( w_3^k (w_1^k)^{-1})$ are independent in $\frac{F'_k/F'_{k+1}}{Y}$ we must first determine the set $A$.

We begin by simplifying the expression for $i_{n*}(y_3)$ using the commutator identity $[a,vb]=[a,v] \lup{v}[a,b]$ as follows:
\begin{align*}
i_{n*}(y_3)&=[c_4, c_5 c_6^n c_3]\\
&=[c_4,c_5] \lup{c_5}[c_4,c_6^nc_3]\\
&=[c_4,c_5] \lup{c_5}[c_4,c_6^n]\lup{c_5c_6^{n}}[c_4,c_3]\\
&=[c_4,c_5] \lup{c_5}[c_4,c_6]\lup{c_5c_6}[c_4,c_6^{n-1}]\lup{c_5c_6^{n}}[c_4,c_3]\\
&=[c_4,c_5] \lup{c_5}[c_4,c_6] \lup{c_5c_6}[c_4,c_6] \cdots \lup{c_5c_6^{n-1}}[c_4,c_6] \lup{c_5c_6^{n}}[c_4,c_3].
\end{align*}
Note that the term $\lup{c_5c_6^{n}}[c_4,c_3]$ is not an element of our chosen basis for $H_1(F', \Z)$ as $c_6>c_4$.  In order to express $i_{n*}(y_3)$ in terms of our basis for $H_1(F',\Z)$ we rewrite this term as follows:

\begin{align*}
\lup{c_5c_6^{n}}[c_4,c_3]=&[c_3,c_4][c_4,c_5][c_3,c_5] \lup{c_4}[c_1,c_5] \lup{c_3}[c_4,c_5] \prod_{i=0}^{n-1} \lup{c_5c_6^i}[c_3,c_6] \\
&\prod_{i=0}^{n-1}\lup{c_5c_6^i}[c_3,c_6]
\prod_{i=0}^{n-1}\lup{c_5c_3c_6^i}[c_4,c_6] \prod_{i=0}^{n-1}\lup{c_5c_4c_6^i}[c_3,c_6].
\end{align*}

Collecting the basis elements of $B$ that occur in the above expressions for $i_{n*}(y_3), n \in N$ we find $A$ to be the following set:
$$
A=\left\{ \left.
\begin{array}{l}
[c_4,c_5], \lup{c_5}[c_4,c_6], \lup{c_5c_6^{i}}[c_4,c_6], \left[c_3,c_4\right],[c_4,c_5],[c_3,c_5], \\
\lup{c_4}[c_1,c_5],\lup{c_3}[c_4,c_5],\lup{c_5c_6^i}[c_3,c_6], \lup{c_5c_3c_6^i}[c_4,c_6],\lup{c_5c_4c_6^i}[c_3,c_6]
\end{array}
\right| i \in \N \right\}
$$

Note that by construction, when viewed as elements of $\frac{F'/F''}{Y}$,
\begin{align*}
i_{n*}(w_3^k(w_1^k)^{-1})&=i_{n*}(j \pi(w_3^k) \eta^k)\\
&=i_{n*}(j \pi(w_3^k)).
\end{align*}
Thus the elements $i_{n*}(w_3^k(w_1^k)^{-1})$ are independent in $F'/F''$ if the elements $i_{n*}(j \pi(w_3^k))$ are independent in $\frac{F'/F''}{Y}$.

Modulo $Y$, $i_{n*}(j \pi (w_3^k))$ can be written as follows:
\begin{align*}
i_{n*}\left(j \pi w_3^k \right)=&-\left[ \cdots \left[\lup{c_3c_4^{-1}c_3^{-1}}[c_5, c_6], [c_2,c_1]\right], \cdots [c_2,c_1] \right]\\
&-\left[ \cdots \left[\lup{c_3c_4^{-1}c_3^{-1}}[c_3, c_5], [c_2,c_1]\right], \cdots [c_2,c_1] \right]\\
&- \sum_{i=0}^{n-1}\left[ \cdots \left[\lup{c_3c_4^{-1}c_5c_6^ic_3^{-1}}[c_3,c_6], [c_2,c_1]\right], \cdots [c_2,c_1] \right]\\
&-\left[ \cdots \left[\lup{c_3c_4^{-1}}[c_4, c_5], [c_2,c_1]\right], \cdots [c_2,c_1] \right]\\
&-\sum_{i=0}^{n-1}\left[ \cdots \left[\lup{c_3c_5c_6^i c_4^{-1}}[c_4,c_6] , [c_2,c_1]\right], \cdots [c_2,c_1] \right].
\end{align*}

Consider a finite sum of these elements.  In $\frac{F'_k/F'_{k+1}}{Y}$, this sum can be written as $\sum_{m=1}^M \alpha_m i_{n_m *} (j \pi (w_3^k))$ where $\alpha_m \ne 0$ and $n_{m}<n_{m+1}$ for all $m$.  The $M^{th}$ term of this product is the only term containing a multiple of the basis element  $\left[ \cdots \left[\lup{c_3c_5c_6^{M-1} c_4^{-1}}[c_6,c_4], [c_2,c_1]\right], \cdots [c_2,c_1] \right]$.  Hence the sum cannot be trivial, and thus the elements $i_{n *} (j \pi (w_3^k))$ must be independent in $\frac{F'_k/F'_{k+1}}{Y}$.  Therefore the elements $i_{n *} (w_3^k(w_1^k)^{-1})$ are independent in $\frac{F'_k}{F'_{k+1}}$.

As the set $\{i_{n*}(w_3^k(w_1^k)^{-1})^{\alpha_n}|n \in \N\}$ is an independent set in $F'_k/F'_{k+1}$, the set $\{(1-c_6)(1-c_2)i_{n*}(w_3^k(w_1^k)^{-1})\}$ is also an independent set.
As $\tau'_k(i_n'(f_k))[c_6,c_2]=(1-c_6)(1-c_2)i_{n*}(w_3^k(w_1^k)^{-1})$, this shows that $\frac{M_k(S)}{M_k(S)}$ surjects onto an infinite rank torsion free abelian subgroup of $F'_k/F'_{k+1}$ via the map $f \mapsto \tau'_k(f)[c_6,c_2]$.

 
\end{proof}


\end{document}